\documentclass[a4paper]{siamart250211}

\usepackage{graphicx}%
\usepackage{multirow}%
\usepackage{amsmath,amssymb,amsfonts}%
\usepackage{mathrsfs}%
\usepackage[title]{appendix}%
\usepackage{xcolor}%
\usepackage{textcomp}%
\usepackage{manyfoot}%
\usepackage{booktabs}
\usepackage{algorithm}%
\usepackage{algorithmicx}%
\usepackage{algpseudocode}%
\usepackage{listings}%
\usepackage{lineno,hyperref}
\usepackage{amscd}
\usepackage{bm}
\usepackage{color}
\usepackage{cleveref}
\usepackage{epstopdf}
\usepackage{float}
\usepackage{lmodern}
\usepackage{subcaption}
\usepackage{todonotes}
\usepackage{nicefrac}
\usepackage{units}
\newcommand\uu{\boldsymbol{u}}

\newcommand\zz{\boldsymbol{z}}
\newcommand\vv{\boldsymbol{v}}
\newcommand\ww{\boldsymbol{w}}
\newcommand\ff{\boldsymbol{f}}
\newcommand\gf{\boldsymbol{g}}
\newcommand\bx{\boldsymbol{x}}
\newcommand\bzero{\boldsymbol{0}}

\newcommand\bs{\boldsymbol{s}}
\newcommand\ppsi{\boldsymbol{\psi}}

\newcommand{\ddiv}{\mathrm{div}\,}

\newcommand\OO{\Omega}
\newcommand{\zero}{\boldsymbol{0}}
\newcommand{\vecb}[1]{\boldsymbol{#1}}
\newcommand\II{\mathbb{I}}
\newcommand{\Pim}{\Pi_{\boldsymbol{V}^-}}
\newcommand{\Pihm}{\Pi_{\boldsymbol{V}^-_h}}

\DeclareMathOperator{\re}{\text{Re}}
\DeclareMathOperator{\im}{\text{Im}}
\newcommand{\Ih}{\mathcal{\bf I}_h}

\newcommand\GU{\Gamma_{d}}

\newcommand\GT{\Gamma_{r}}
\newcommand\UU{\boldsymbol U}
\newcommand\VV{\boldsymbol V}
\newcommand\ZZ{\boldsymbol Z}
\newcommand\RR{\mathbb R}
\newcommand{\nnormOne}[1]{\Vert #1\Vert_{1}}
\newcommand{\nnormZero}[1]{\Vert #1\Vert_{0}}
\newcommand{\nnormDiv}[1]{\Vert #1\Vert_{\mathrm{div}}}
\newcommand{\uprod}[2]{\left(#1,#2\right)}
\newcommand\uve{\vec{\uu}}
\newcommand\vve{\vec{\vv}}

\newtheorem{remark}[theorem]{Remark}
\newtheorem{assumption}[theorem]{Assumption}

\DeclareMathOperator{\Span}{span}
\newcommand\Trih{{\mathcal T}_h}
\renewcommand{\(}{\left(}
\renewcommand{\)}{\right)}

\ifpdf
  \DeclareGraphicsExtensions{.eps,.pdf,.png,.jpg}
\else
  \DeclareGraphicsExtensions{.eps}
\fi

\headers{Wave propagation in thermo-poroelasticity}{H. Li, C. Carcamo, H. Rui, and  V. John}

\title{Coupling of conforming and mixed finite element methods for a model of wave propagation in thermo-poroelasticity in the frequency domain
\thanks{Submitted to the editors.
\funding{The work of Hongpeng Li and Hongxing Rui was supported by the National Natural Science Foundation of China (No. 12131014).}}}

\author{Hongpeng Li \thanks{School of Mathematics, Shandong University, Jinan 250100, China (\email{lihongpeng\_sd@163.com}).}
\and Cristian C\'arcamo \thanks{Weierstrass Institute for Applied Analysis and Stochastics, Anton-Wilhelm-Amo Str.~39, 10117 Berlin, Germany  (\email{cristian.carcamosanchez@wias-berlin.de}).}
\and Hongxing Rui \thanks{School of Mathematics, Shandong University, Jinan 250100, China (\email{hxrui@sdu.edu.cn}).}
\and Volker John \thanks{Weierstrass Institute for Applied Analysis and Stochastics, Anton-Wilhelm-Amo Str.~39, 10117 Berlin, Germany and Freie Universit\"at of Berlin, Department of Mathematics and Computer Science, Arnimallee 6, 14195 Berlin, Germany (\email{volker.john@wias-berlin.de}).}
}

\allowdisplaybreaks

\begin{document}
	
	\maketitle
	
	\begin{abstract}
		A dynamic linear thermo-poroelasticity model, containing inertial and relaxation terms with second-order time derivatives, is investigated 
		in this paper. 
		The mathematical and numerical analysis of this model is performed in the frequency domain.
		The variational formulation is analyzed within the framework of Fredholm’s alternative and T-coercivity. 
		Under appropriate assumptions on the coefficients, the well-posedness of the problem is proved. For its discretization, 
		we propose a stabilized coupling of conforming and mixed finite element spaces, which are free of volumetric locking, and pressure as well as temperature oscillations.
		By incorporating projections in certain sesquilinear forms, the well-posedness of the finite element solution can be obtained through  
		a similar reasoning as in the continuous case.
		Optimal error estimates are derived for all variables.
		Numerical studies validate the accuracy and robustness of the proposed method.
	\end{abstract}

	\begin{keywords}
		thermo-poroelasticity,  frequency domain, wave propagation, locking-free and oscillation-free mixed finite element methods, well-posedness, convergence
	\end{keywords}
	
	\begin{MSCcodes}
		65N12, 74F05, 74J05, 74S05
	\end{MSCcodes}
	
\section{Introduction}

The studied thermo-poroelasticity model describes the coupling between mechanical deformation, fluid flow, and heat transfer in porous media. 
As an extension of the isothermal poroelasticity \cite{biot1941general}, this model incorporates the influence of temperature on mechanical displacement and fluid flow.
Modeling thermo-poroelasticity is of great significance in many applications, including geothermal energy systems, carbon dioxide sequestration, petroleum engineering \cite{wang2003coupled}, and biomedical engineering \cite{Ge2022new}.

Biot \cite{Biot1956Thermoelasticity} first introduced a wave propagation model in poroelastic media which includes thermal effects. 
However, using only the Fourier law of heat conduction, this model might give nonphysical results, such as an infinite diffusion velocity.
Lord and Shulman \cite{lord1967generalized} incorporated a relaxation term into the heat equation, thereby leading to a hyperbolic system and avoiding infinite velocity. 
Carcione et al.\cite{carcione2019simulation} first simulated wave propagation in thermo-poroelastic media with physically realistic velocities by introducing a relaxation term, based on the theory of Lord and Shulman.
By using the Fourier transform of the thermodynamic equations, Wang et al.\cite{wang2020green} developed a frequency-domain Green's function that describes the displacement and temperature response to an elastic or thermal point source.

Constructing numerical schemes for thermo-poroelasticity is highly challenging due to its complex coupling properties. 
From this perspective, several analytical frameworks and numerical schemes for (thermo-)poroelasticity have been established in the literature. 
Recent studies, such as \cite{Santos2021Existence}, demonstrated the existence and uniqueness of wave propagation in linear thermo-poroelastic isotropic media, and proved its regularity in both space and time variables.
Barr{\'e} et al.\cite{barre2022analysis} established the existence of weak solutions for nearly incompressible poromechanical materials
using a semigroup approach.
Discontinuous Galerkin (DG) methods for wave propagation phenomenon in thermo-poroelasticity and poro-elasto-acoustic coupled problem are detailed in \cite{Antonietti2022high,JCPBonettiNumerical}.
Zhao \cite{zhao2025staggered} proposed a staggered DG scheme to simulate wave propagation in poroelasticity on general polygonal meshes, which is robust with respect to grid shapes and hanging nodes.
Finally, we mention some recent studies on quasi-static thermo-poroelasticity, which neglects inertial and relaxation effects, but contains a nonlinear convective term in the heat equation.
Various numerical methods of this model 
have been proposed and studied, 
including the enriched Galerkin method \cite{Yi2024Physics}, the multiphysics finite element method \cite{Ge2022Multiphysics}, 
splitting and iterative decoupling techniques \cite{Brun2020Monolithic,Kolesov2014Splitting}, DG methods \cite{AntoniettiDiscontinuous}, and preconditioning techniques \cite{cai2025parameter}.

The fully dynamic thermo-poroelasticity describes the wave propagation in the thermo-poroelastic media.
Biot's theory \cite{Biot1962Mechanics} for fluid-saturated poroelastic media predicted the existence of two compressional (P) waves, one fast and the other one slow, and one shear wave, without taking into account the temperature. The slow wave is diffusive at low frequencies and propagative at high frequencies.
Rudgers \cite{rudgers1990analysis} described the behavior of thermoacoustic waves, in particular velocities and absorption coefficients as functions of frequency. Further, coupling with temperature yields an additional slow thermal (T) wave \cite{sharma2008wave}. 
It exhibits a diffusive character at low thermal conductivities and a wave-like one at high values \cite{wang2021generalized,wang2020green}.
Therefore, it is highly significant to study the wave propagation in thermo-poroelasticity. 
As it is well known, constructing numerical schemes in the frequency domain can significantly enhance efficiency compared with methods in the time domain \cite{arbenz2017comparison,mcmullen2003application}. The smooth temporal variations in thermo-poroelastic boundary conditions can be captured by few Fourier modes, making frequency-domain simulation orders of magnitude more efficient than time-stepping. 
Further, frequency-domain simulations do not need to take into account initial conditions, parallel computation over time, and the stability of time integration.


This work addresses the mathematical analysis of wave propagation in thermo-poroelasticity in the harmonic regime.
To the best of our knowledge, this is the first time to present a functional framework for the wave propagation in thermo-poroelasticity in the frequency domain, including the proof of well-posedness and convergence.
One of the main difficulties of this contribution lies in proving the well-posedness of the variational and discrete formulations.
In the weak formulation corresponding to the thermo-poroelasticity in harmonic regime, some sesquilinear forms with frequency number are not coercive and indefinite. The so-called Banach--Ne$\check{c}$as--Babu$\check{s}$ka (BNB) Theorem is not applicable. Furthermore, this model is large in scale, and the inner product only satisfies the conjugate symmetry.
$\mathsf{T}$-coercivity \cite{ciar12} and the G\r{a}rding inequality provide an alternative to the inf-sup condition, which has been applied to solve Helmholtz-type problems.
Another challenge concerns the discrete formulation, which must overcome locking phenomenon in the thermo-poroelasticity, including volumetric locking, pressure and temperature oscillations. 
Volumetric locking means that the solid skeleton is nearly incompressible as the Lam\'{e} constant $\lambda\rightarrow\infty$, i.e., the deformation is in a divergence-free state, and the rate of convergence for the solid displacement deteriorates.
Pressure oscillations occur when the constrained storage coefficient $c_0$ and the thermal dilatation coefficient $b_0$ are zero, 
while temperature oscillations occur when $b_0$ and the thermal capacity $a_0$ are zero.
A widely accepted explanation for this pressure instability in poroelasticity is that, if $c_0=0$ and the permeability of the porous media is extremely low, the solid phase behaves as an incompressible medium in the early time \cite{phillips2009overcoming}. In 
\cite{Yi2017Study} it was pointed out that pressure oscillations occur due to the incompatibility of the finite element spaces for the solid displacement and the pore pressure.
From the perspective of modeling, there exists an oscillation of the temperature when the effective thermal conductivity is low, 
for reasons analogous to the underlying pressure oscillations.

This paper studies the wave propagation in fully-dynamic thermo-poroelasticity in the frequency domain. 
The corresponding variational formulation in the harmonic regime is proposed. By Fredholm's alternative, we demonstrate the well-posedness under reasonable assumptions on the thermo-poroelastic parameters. Particularly, by exploiting the compact embedding property, we express the differential operator corresponding to the model problem as a bijective operator and a compact perturbation. 
The Bernardi--Raugel element is used to approximate the solid displacement, while incorporating projections in certain sesquilinear forms, in order to obtain parameter robustness with respect to $\lambda$ and preserve the compactness property, thus ensuring the well-posedness of the finite element problem. 
The filtration displacement and pore pressure are approximated by the Raviart--Thomas element, while the temperature is approximated by the piecewise linear polynomial space. 
Pressure oscillations can be avoided since the finite element spaces for the solid displacement and pore pressure satisfy an inf-sup condition.
To avoid temperature oscillations when the discrete system is close to a saddle point problem, we
introduce an additional stabilization term in the temperature equation, which has the form of an 
artificial thermal conductivity.

The remainder of the paper is organized as follows. 
Section~\ref{sec:model} provides details of the model and the weak formulation.
Section~\ref{sec:continuous-analysis} is devoted to proving the well-posedness of the weak formulation.
In Section~\ref{sec:discrete-analysis}, we construct finite element formulation, analyze its well-posedness and the rate of convergence.
We present some numerical studies with analytical solutions and benchmark problems in Section~\ref{sec:numericalExperiments}.
The paper closes with a summary and outlook.

\section{Model Problem}\label{sec:model}
\subsection{Thermo-poroelasticity in the harmonic regime}

Our model, introduced in \cite{JCPBonettiNumerical,Santos2021Existence}, comprises a momentum conservation equation, a mass conservation 
equation, and an energy conservation equation. Therein, the inertial terms and relaxation terms are included to obtain physically consistent results. Assume that $\Omega \subset\mathbb{R}^d$, $d \in \{2,3\}$, is a bounded domain with polyhedral and Lipschitz-continuous boundary $\partial\Omega$ and $T_f>0$ is a
final time. The problem reads, find $(\uu,\ww,p,T)$ defined
on $\Omega\times (0,T_f]$ such that:
\begin{equation} \label{eq:TherPoro}
	\left.\begin{array}{rcl}
			\rho \ddot{\uu}+\rho_f \ddot{\ww}-\ddiv{\boldsymbol{\sigma}}&=&\tilde{\ff}, \\
			\rho_f \ddot{\uu}+\rho_w \ddot{\ww}+\boldsymbol{K}^{-1} \dot{\boldsymbol{w}}+\nabla p&=&\tilde{\gf},\\
			c_0 \dot{p}-b_0 \dot{T}+\alpha\ddiv\dot{\boldsymbol{u}}+\ddiv\dot{\boldsymbol{w}}&=&0,\\
			a_0(\dot{T}+\tau \ddot{T})-b_0(\dot{p}+\tau \ddot{p})+\beta(\ddiv\dot{\uu}+\tau\ddiv\ddot{\uu})
			 -\ddiv(\boldsymbol{\Theta} \nabla T)&=&\tilde{H}
	\end{array}\right\}
\end{equation}
in $\Omega\times (0,T_f]$.
The dots above unknowns represent the first and second temporal derivatives.
This model contains two types of displacement: the solid displacement $\uu\ [\unit{m}]$ and the filtration displacement $\ww\ [\unit{m}]$, in conjunction with the pore pressure 
$p\ [\unit{Pa}]$ and temperature distribution $T\ [\unit{K}]$ as primary unknowns.
The filtration displacement refers to the relative displacement of the solid phase with respect to the fluid phase, with the porosity as a scaling factor.
Thus, $\dot{\ww}$ can be regarded as a new unknown, the volumetric fluid velocity \cite{barre2022analysis}. 
The second equation provides the dynamic formulation of Darcy's law. 
This system can be seen as a combination of Darcy, Brinkman, and Biot equations.
The terms $\tilde{\ff}\ [\unitfrac{N}{m^3}]$, $\tilde{\gf}\ [\unitfrac{N}{m^3}]$, and $\tilde{H}\ [\unitfrac{Pa}{(K\ s)}]$ denote the applied body force, the volumetric fluid, and heat source/sink. 
The density of the mixture $\rho$ is defined as $\rho=\phi\rho_f+(1-\phi)\rho_s$, and $\rho_w=\nicefrac{a \rho_f}{\phi}$, 
with the porosity $\phi$ and tortuosity $a$ satisfying $0<\phi_0\leq\phi\leq\phi_1<1$ and $a>1$.
The total stress tensor $\boldsymbol{\sigma}$ satisfies the constitutive equation 
$\boldsymbol{\sigma}(\uu,p,T)=2\mu\varepsilon(\uu)+\lambda(\ddiv \uu)\boldsymbol{I}-\alpha p\boldsymbol{I}-\beta T\boldsymbol{I}$,
where $\varepsilon(\uu)=\nicefrac{1}{2}(\nabla\uu+\nabla\uu^T)$ is the symmetric strain tensor 
and $\boldsymbol{I}\in\mathbb{R}^{d\times d}$ is the identity matrix. 
Other constitutive laws can be found, e.g., in \cite{barre2022analysis,JCPBonettiNumerical}.
The physical meaning and assumptions of the model's coefficients are provided in Table~\ref{Tab:modelCoefficients} and Assumption \ref{model coefficients}.

\begin{table}[htbp]
	\centering
	\caption{Thermo-poroelastic coefficients appearing in model \eqref{eq:TherPoro}}
	\label{Tab:modelCoefficients}
	\begin{tabular}{cll}
		\toprule  
		Notation                &Quantity                                    &Unit \\ 
		\midrule  
		$a_0$                     &thermal capacity                                  &$\unitfrac{Pa}{K^2}$ \\
		$b_0$                     &thermal dilatation coefficient                    &$\unitfrac1{K}$ \\
		$c_0$                     &specific storage coefficient                      &$\unitfrac1{Pa}$ \\
        $\lambda$, $\mu$          &Lam\'{e} coefficients                             & \unit{Pa} \\
		$\alpha$                  &Biot--Willis constant                              & - \\
		$\beta$                   &thermal stress coefficient                        &$\unitfrac{Pa}{K}$ \\
		$\rho_s$                  &solid matrix density                              &$\unitfrac{kg}{m^3}$ \\ 
		$\rho_f$                  &saturating fluid density                          &$\unitfrac{kg}{m^3}$ \\
		$\boldsymbol{K}$          &permeability divided by fluid viscosity           &$\unitfrac{m^2}{(Pa\, s)}$ \\  
		$\boldsymbol{\Theta}$     &effective thermal conductivity                    &$\unitfrac{m^2\, Pa}{(K^2\, s)}$ \\
		$\phi$                    &porosity                                          & - \\
		$a$                       &tortuosity                                        & - \\
		$\tau$                    &Maxwell--Vernotte--Cattaneo relaxation time         & \unit{s} \\
		\bottomrule 
	\end{tabular}
\end{table}

Mathematical analysis and numerical simulations require a dimensionless problem. For the sake of 
simplifying the notation, we assume that this problem has the form \eqref{eq:TherPoro} with the same notations
for all functions and coefficients. 

\begin{assumption}[Model coefficients] \label{model coefficients}
	\begin{itemize}
        \item[1.] The constant $\alpha$ satisfies $\phi<\alpha\leq 1$, and $\beta$ is strictly positive.
		\item[2.] The constants $a_0$, $b_0$ and $c_0$ satisfy $a_0\geq b_0^2c_0^{-1}$, $b_0\geq 0$, and $c_0> 0$.
        \item[3.] The Lam\'{e} coefficients $\lambda$ and $\mu$ satisfy $0<\lambda_0<\lambda<\infty$ and $0<\mu_1<\mu<\mu_2$.
		\item[4.] The tensors $\boldsymbol{K}(\bx)$ and $\boldsymbol{\Theta}(\bx)$ are symmetric, positive-definite.
		There exist positive constants $k_{\min}$, $k_{\max}$ and $\theta_{\min}$, $\theta_{\max}$ such that for any $\bx\in\Omega$,
		\begin{eqnarray}\label{eq:tensor_K}
			k_{\min}|\boldsymbol{\xi}|^2&\leq&\boldsymbol{\xi}^T\boldsymbol{K}(\bx)\boldsymbol{\xi}\leq k_{\max}|\boldsymbol{\xi}|^2, \quad\forall\ \boldsymbol{\xi}\in\mathbb{R}^d, \\
            \label{eq:tensor_theta}
			\theta_{\min}|\boldsymbol{\xi}|^2&\leq&\boldsymbol{\xi}^T\boldsymbol{\Theta}(\bx)\boldsymbol{\xi}\leq \theta_{\max}|\boldsymbol{\xi}|^2, \quad\forall\ \boldsymbol{\xi}\in\mathbb{R}^d.
		\end{eqnarray}
	\end{itemize}
\end{assumption}

Inspired by \cite{Cristian2024}, taking the Fourier transform in time, we obtain the model \eqref{eq:TherPoro} in the harmonic 
regime for a given positive frequency $\omega$:
\begin{equation*}
	\begin{array}{rcl}
			-\omega^2\rho\uu-\omega^2\rho_f\ww-\ddiv\boldsymbol{\sigma}&=&\tilde{\ff}, \\
			-\omega^2\rho_f\uu-\omega^2\rho_w\ww+i\omega\boldsymbol{K}^{-1}\boldsymbol{w}+\nabla p&=&\tilde{\gf}, \\
			i\omega c_0p-i\omega b_0T+i\omega\alpha\ddiv\uu 
			+i\omega\ddiv\boldsymbol{w}&=&0,\\
			a_0(i\omega-\omega^2\tau)T-b_0(i\omega-\omega^2\tau)p+\beta(i\omega-\omega^2\tau)\ddiv\uu
			-\ddiv(\boldsymbol{\Theta} \nabla T)&=&\tilde{H}.
	\end{array}
\end{equation*}
In case no ambiguity occurs, we use the same notations as in \eqref{eq:TherPoro} for the complex-valued $\omega$-Fourier modes of unknowns 
and right-hand sides, and $i$ denotes the imaginary unit.
Each complex-valued function can be decomposed into its real and imaginary parts, e.g., $\uu=(\re \uu, \im \uu)$. 
For the convenience of further analysis, in the third equation we eliminate $i\omega$, and the fourth equation is multiplied by $\nicefrac{i}{(i\omega-\omega^2\tau)}$.
The final thermo-poroelasticity model in the frequency domain reads:
\begin{equation} \label{eq:TherPoro-fft}
	\left.\begin{array}{rcl}
			-\omega^2\rho\uu-\omega^2\rho_f\ww-\ddiv\boldsymbol{\sigma}&=&\ff, \\
			-\omega^2\rho_f\uu-\omega^2\rho_w\ww+i\omega\boldsymbol{K}^{-1}\boldsymbol{w}+\nabla p&=&\gf, \\
			c_0p - b_0T + \alpha\ddiv\uu + \ddiv\boldsymbol{w}&=&0,\\
			ia_0 T-ib_0 p+i\beta\ddiv\uu-\nicefrac{i}{(i\omega-\omega^2\tau)}\ddiv(\boldsymbol{\Theta} \nabla T)&=&H.
	\end{array}\right\}
\end{equation}
We assume the homogeneous boundary conditions, the nonhomogeneous case follows via classical lifting arguments.
Following the literature for the thermo-poroelasticity model, we consider three partitions
$\{\Gamma_d, \Gamma_t\}$, $\{\Gamma_p, \Gamma_f\}$, 
and $\{\Gamma_r, \Gamma_h\}$ of $\partial\Omega$ with $|\Gamma_d|>0$, $|\Gamma_r|>0$, 
$\partial \Omega=\overline{\Gamma_d} \cup \overline{\Gamma_t}=\overline{\Gamma_p} \cup \overline{\Gamma_f}=\overline{\Gamma_r} \cup \overline{\Gamma_h}$. 
The boundary conditions read:
\begin{equation*}
\begin{array}{rr}
	\boldsymbol{u}=\boldsymbol{0} \text { on } \Gamma_d, & {\boldsymbol{\sigma}} \boldsymbol{n}=\boldsymbol{0} \text { on } \Gamma_t, \\
	\boldsymbol{K} \boldsymbol{w} \cdot \boldsymbol{n}=0 \text { on } \Gamma_p, & p=0 \text { on } \Gamma_f, \\
	T=0 \text { on } \Gamma_r, & \quad \boldsymbol{\Theta}\nabla T\cdot\boldsymbol{n}=0 \text { on } \Gamma_h.
\end{array}
\end{equation*}
Notice that the second equation in \eqref{eq:TherPoro} only holds under a frequency constraint \cite{Antonietti2022high,Chiavassa2013},
which reads,
\begin{equation} \label{remark:low frequency}
\omega_{\mathrm{crit}}:=\frac{\phi}{2\pi a k_{\mathrm{max}}\rho_f} = \frac1{2\pi a k_{\mathrm{max}}\rho_w}.
\end{equation}
It can be seen in Section~\ref{sec:continuous-analysis} that the analysis also has to take into account
this critical value.

\subsection{Weak formulation}
Throughout this paper, we use standard notations for Lebesgue and Sobolev spaces. 
Let $H^k(\Omega)$, $k\in\mathbb{Z}$, be the Hilbert space, 
and the space $H^k(\ddiv,\Omega)$, $k\in\mathbb{Z}$, is defined by 
\[
H^k(\ddiv,\Omega) := \left\{\vv \in H^k(\Omega)^d\ : \ \ddiv\vv\in H^k(\Omega)\right\},
\]
and $H^0(\ddiv,\Omega)$ coincides with $H(\ddiv,\Omega)$.
The inner product of $L^2(\Omega)$ of complex-valued functions is given by 
	$(u,v) := \int_{\OO} u\,\overline{v}\mathrm{~dx}$,
where $\overline{v}$ denotes the complex conjugate of $v$. It satisfies the conjugate symmetry $(u,v)=\overline{(v,u)}$.
And the inner products in $H^1(\Omega)$ and $H(\ddiv;\Omega)$ are given by 
\begin{align*}
	(u,v)_{1} := (u,v)+(\nabla u, \nabla v), \quad (\uu,\vv)_{\ddiv} := (\uu,\vv)+(\ddiv \uu, \ddiv \vv).
\end{align*}
We denote the induced norms and seminorms by $\Vert \,\cdot\,\Vert_0$, $\Vert \,\cdot\,\Vert_1$, $\vert \,\cdot\,\vert_1$ and $\Vert \,\cdot\,\Vert_{\ddiv}$, respectively. 
Define the following complex-valued spaces for the solid displacement $\uu$, the filtration displacement $\ww$, 
the pore pressure $p$, and the temperature $T$, respectively:
\begin{equation*}
	\begin{aligned}
			\VV & :=\{\vv\in H^1(\OO)^d:\vv=\zero\text{ on }\GU\},\\
			\ZZ & :=\{\zz\in H(\ddiv,\OO):\zz\cdot\boldsymbol{n}=\zero\text{ on }\Gamma_p\},\\
		      Q & := L^2(\OO), \\
		      S & := \{q\in H^1(\OO): q=0   \text{ on }\GT\}.\\
	\end{aligned}
\end{equation*}
Since it is assumed that $|\GU| >0$, a Poincar\'e--Steklov inequality and a Korn inequality 
\begin{equation}\label{eq:korn}
\Vert \vv\Vert_0 \le C_p \Vert\nabla\vv\Vert_0, \quad \Vert \nabla\vv \Vert_0 \le C_{k} \Vert \varepsilon(\vv) \Vert_0, \quad \forall \ \vv \in \VV
\end{equation}
with $C_p, C_k>0$ hold true, e.g., see Eq.~(42.9), Theorems~42.10, 42.11 in \cite{ErnFinite}.
The product space $\UU:=\VV\times \ZZ\times Q\times S$ is endowed with the norm
\begin{equation*}
	\Vert (\vv, \zz, q, s) \Vert^2_{\UU}:= 		
	2\mu\nnormZero{\varepsilon(\vv)}^2 + \lambda\nnormZero{\ddiv\vv}^2 + \Vert\zz\Vert^2_{\ddiv} +\nnormZero{q}^2 + \nnormOne{s}^2.
\end{equation*}
Define the following sesquilinear forms:
\begin{alignat}{4}
&a_1: \VV\times\VV\to\mathbb{C}, \ && a_1(\uu,\vv) &&:= &&\ -\omega^2\rho\uprod{\uu}{\vv} 
	+ 2\mu\uprod{\varepsilon(\uu)}{\varepsilon(\vv)}\label{eq:a_1_form} \\
& && && &&  \   + \lambda\uprod{\ddiv\uu}{\ddiv\vv},\nonumber \\
&a_2: \ZZ\times\ZZ\to\mathbb{C}, \ &&a_2(\ww,\zz)	&&:= &&\  -\omega^2\rho_w\uprod{\ww}{\zz}  + i\omega \uprod{\boldsymbol{K}^{-1}\ww}{\zz},  \label{eq:a_2_form}\\
&a_3: S\times S\to\mathbb{C}, \ &&a_3(T,s) &&:= &&\  a_0i\uprod{T}{s} 
	+ i/(i\omega-\omega^2\tau)\uprod{\boldsymbol{\Theta}\nabla T}{\nabla s}, \label{eq:a_3_form}\\
&b_1: \ZZ\times\VV\to\mathbb{C}, \ &&b_1(\ww,\vv) &&:= &&\  -\omega^2\rho_f\uprod{\ww}{\vv}, \label{eq:b1_form}\\
& b_2: Q\times\VV\to\mathbb{C}, \ &&b_2(p,\vv) &&:= && \  -\alpha\uprod{p}{\ddiv\vv}, \label{eq:b2_form}\\
&b_3: Q\times\ZZ\to\mathbb{C}, \ &&b_3(p,\zz) &&:= && \ -\uprod{p}{\ddiv\zz}, \label{eq:b3_form}\\ 
&c_1: S\times\VV\to\mathbb{C}, \ &&c_1(T,\vv) &&:= &&\  -\beta\uprod{T}{\ddiv\vv}, \label{eq:c1_form} \\
&c_2: S\times Q\to\mathbb{C}, \ &&c_2(T,q) &&:= && \ -b_0\uprod{T}{q}, \label{eq:c2_form} \\
&d: Q\times Q\to\mathbb{C}, \ &&d(p,q) &&:= &&\  c_0\uprod{p}{q}. \label{eq:d_form}
\end{alignat}
The adjoint sesquilinear forms are denoted with a superscript star, e.g., $b_1^*$.

The variational formulation of \eqref{eq:TherPoro-fft}, which is derived in the standard way, reads as follows:
Find $\uve = (\uu,\ww,p,T)\in\UU$ such that
\begin{equation}\label{weakF}
	\langle \mathcal{A}(\uve),\vve \rangle = \langle \mathcal{F},\vve \rangle \quad\forall\ \vve = (\vv,\zz,q,s)\in\UU,
\end{equation}
where the operators $\mathcal{A}:\UU\to\UU^*$ and $\mathcal F \in \UU^*$ read
\begin{eqnarray*} 
	\nonumber	\langle \mathcal{A}(\uve),\vve \rangle & :=& 
		   a_1(\uu,\vv) + b_1(\ww,\vv) + b_2(p,\vv) + c_1(T,\vv) + b_1^*(\uu,\zz) + a_2(\ww,\zz) \nonumber  \\
	\nonumber	 && + b_3(p,\zz)  - b_2^*(\uu,q) - b_3^*(\ww,q) + d(p,q) + c_2(T,q) - i c_1^*(\uu,s)\nonumber\\
	&&+ i c_2^*(p,s) + a_3(T,s), \\
	\nonumber	\langle \mathcal{F},\vve \rangle &:=& \uprod{\ff}{\vv} + \uprod{\gf}{\zz} + \uprod{H}{s}.
\end{eqnarray*}

	\section{Analysis of the Continuous Problem} \label{sec:continuous-analysis}
In order to prove the well-posedness of the variational formulation \eqref{weakF}, 
some linear bounded sub-operators will be introduced to decompose $\mathcal{A}$:
Define $\mathcal{A}_{1,2}:\VV\times\ZZ \to (\VV\times\ZZ)^*$, $\mathcal{B}:Q\to(\VV\times\ZZ)^*$, 
$\mathcal{C}_1:S\to(\VV\times\ZZ\times Q)^*$, $\mathcal{C}_2:\VV\times\ZZ\times Q\to S^*$, 
$\mathcal{D}: Q \to Q^*$, and $\mathcal{A}_3: S\to S^*$. 
For simplicity of notation, we use in this section the symbols $\uve, \vve$, which were 
introduced for the space $\UU$, also for subspaces, where in the following, it has to be understood that the respective components are taken. 
Then, the sub-operators read as follows:
\begin{eqnarray}
	\langle\mathcal{A}_{1,2}(\uve),\vve\rangle  &:=& a_1(\uu,\vv) + a_2(\ww,\zz) + b_1(\ww,\vv) 
	+ b_1^*(\uu,\zz), \nonumber \\ 
	\langle\mathcal{B}(\uve),\vve\rangle  &:=& b_2(p,\vv) + b_3(p,\zz),\nonumber\\
	\langle\mathcal{C}_{1}(\uve),\vve\rangle  &:=& c_1(T,\vv) + c_2(T,q), \label{eq:op-C1}  \\
	\langle\mathcal{C}_2(\uve),\vve\rangle &:=& -i c_1^*(\uu,s) + i c_2^*(p,s), \label{eq:op-C2} \\
	\langle\mathcal{D}(\uve),\vve\rangle  &:=& d(p,q),\nonumber\\
	\langle\mathcal{A}_3(\uve),\vve\rangle &:=& a_3(T,s). \nonumber 
\end{eqnarray}
These operators allow us to define a sub-operator $\mathcal{A}_{\mathrm{sub}}$ and rewrite $\mathcal{A}$ in the following form:
\begin{equation}\label{eq:oper-decomp}
	\mathcal{A}_{\mathrm{sub}} = 
	\begin{pmatrix}
		\mathcal{A}_{1,2} & \mathcal{B}  \\ -\mathcal{B}^* & D
	\end{pmatrix},\quad
	\mathcal{A} = 
	\begin{pmatrix}
		\mathcal{A}_{\mathrm{sub}} & \mathcal{C}_1  \\ \mathcal{C}_2 & \mathcal{A}_3
	\end{pmatrix}.\,
\end{equation}
This decomposition will be utilized to establish the bijectivity
of $\mathcal{A}_{\mathrm{sub}}$ as well as $\mathcal{A}$
by employing the G\r{a}rding inequality and Fredholm's alternative.

\subsection{Preliminaries}\label{ssec:preliminaries}

The main difficulty of well-posedness comes from the fact that the underlying sesquilinear forms, 
particularly $a_1(\cdot,\cdot)$, are not sign-defined. 
To address this, $\mathsf{T}$-coercivity is employed as an alternative to the BNB theorem.
The upcoming definition and lemmas establish the key properties of $\mathsf{T}$-coercivity.
\begin{definition}[$\mathsf{T}$-coercivity]\label{def:T-coer}
    Let $V$ and $W$ be two Hilbert spaces, and $a(\cdot,\cdot)$ be a continuous sesquilinear form over $V\times W$. 
    It is $\mathsf{T}$-coercive if there exists $\mathsf{T}\in\mathcal{L}(V, W)$, bijective, and a constant $\tilde{\alpha} >0$, such that
    \[
    \vert a(v,\mathsf{T}v) \vert \geq \tilde{\alpha} \Vert v \Vert^2_V \quad \forall\ v \in V.
    \]
\end{definition}

\begin{lemma}[Well-posedness of $\mathsf{T}$-coercive sesquilinear form]\label{lem:T-coer}
    Under the condition of Definition \ref{def:T-coer}, let $\ell\in W^{\prime}$, the variational problem $a(v,w)=\ell(w)$
     is well-posed, if and only if $a(\cdot,\cdot)$ is $\mathsf{T}$-coercivity in the sense of Definition \ref{def:T-coer}.
\end{lemma}
We refer to \cite{ciar12,ciarlet2025explicit} for the proof.

\begin{lemma}[The Fredholm alternative of index zero] \label{the-Fredholm-alternative}
	Suppose $H$ is a Hilbert space, $A:H\rightarrow H$ is invertible, and $K:H\rightarrow H$ is compact. 
	If $\ker(A+K)=0$, then $(A+K)u=f$ has a unique solution for all $f$.
\end{lemma}
For the proof, see Theorem~8.2 in \cite{SayasVariational}.
The Fredholm alternative states that in case the operator associated with the differential problem can be decomposed into a bijective
operator and a compact perturbation, then this operator is invertible if and only if it is injective.

\subsection{Well-posedness}\label{ssec:cont-stab}
Define inner products in $\VV$ as follows:
\begin{eqnarray*}
    (\uu,\vv)_{0,\rho} &:=& \rho(\uu,\vv), \quad (\uu,\vv)_{1,\mu} := 2\mu(\varepsilon(\uu),\varepsilon(\vv)), \\
    (\uu,\vv)_{1,\mu,\lambda} &:=& (\uu,\vv)_{1,\mu} + \lambda(\ddiv\uu,\ddiv\vv), \\
	(\uu,\vv)_{1,\mu,\lambda,\rho} &:=&  (\uu,\vv)_{1,\mu,\lambda} + (\uu,\vv)_{0,\rho}, \quad \forall\ \uu, \vv\in\VV,
\end{eqnarray*}
and denote the associated norms by $\Vert\vv\Vert_{0,\rho}$, $\Vert\vv\Vert_{1,\mu}$, $\Vert \vv\Vert_{1,\mu,\lambda}$, and $\Vert\vv\Vert_{1,\mu,\lambda,\rho}$.
Since the embedding $\VV \hookrightarrow L^2(\OO)^d$ is compact, which is a case of the Rellich--Kondrachov theorem, e.g., see Theorem~6.2 in \cite{Ada75}, there exists an orthogonal Hilbert basis of $\VV$, with respect to $(\cdot,\cdot)_{1,\mu,\lambda}$. This basis will be fixed by composing it of the eigenvectors of the elasticity operator.
Thus, let $ \{\ppsi_n,\kappa_n\}_n \in \VV\times\RR^+$ be a family of the basis and the corresponding eigenvalues, 
i.e., they satisfy 
\begin{equation}\label{eq:eigenvalues}
		(\ppsi_n,\vv)_{1,\mu,\lambda} = \kappa_n\, (\ppsi_n,\vv)_{0,\rho} \quad \forall \ \vv\,\in\,\VV, \quad
		\lim_{n \to \infty} \kappa_n = +\infty.
\end{equation}
Notice that for $\vv = \ppsi_m$, $m\neq n$, one obtains 
\begin{equation}\label{eq:L2_ortho}
0 = (\ppsi_n,\ppsi_m)_{1,\mu,\lambda} = \kappa_n\, (\ppsi_n,\ppsi_m)_{0,\rho}.
\end{equation}
Hence, the basis functions are also orthogonal with respect to the $L^2$ inner product, and consequently, 
with respect to the inner product $(\cdot,\cdot)_{1,\mu,\lambda,\rho}$. 
Using the latter property and the normalization $ \Vert \ppsi_n \Vert_{1,\mu,\lambda,\rho} = 1$, 
then every $\vv\in\VV$ admits a unique expansion
\begin{equation} \label{eq:expansion_efct}
\vv = \sum_{n\geq 0} \alpha_n \ppsi_n \quad \mbox{with} \quad
\alpha_n  := (\vv,\ppsi_n)_{1,\mu,\lambda,\rho} \quad \mbox{and} \quad \Vert \vv\Vert_{1,\mu,\lambda,\rho}^2 =  \displaystyle \sum_{n\geq 0} \alpha^2_n.
\end{equation}

\begin{assumption}[Well-posedness of the underlying elasticity problem] \label{assum:tcoer}
	Let $\kappa_n$ be the eigenvalues introduced in \eqref{eq:eigenvalues}. We assume that $\omega^2 \notin (\kappa_n)_{n\geq 0}$.
\end{assumption}
In the remainder of the paper, it is always assumed that Assumption~\ref{assum:tcoer} is satisfied, 
without stating it explicitly at each occasion. 
Assumption~\ref{assum:tcoer} guarantees that the considered frequency $\omega$ is not an eigenfrequency. 
Let $\{\ppsi_n\}$ be the set of eigenvectors introduced in \eqref{eq:eigenvalues}. 
We fix a value $\overline{m}\in\mathbb{N}$ by  $\overline{m} = \max\{ n \in \mathbb N \mid \omega^2 > \kappa_n\}$
and define the related finite-dimensional subspace
\begin{equation*}
\VV^-:= \Span_{0\leq n\leq \overline{m}} \{\ppsi_n\}\,.
\end{equation*}
Let $\Pim$ denote the orthogonal projection on $\VV^-$ and let $\mathbb{T} := \II_{\VV} - 2\Pim$, 
where $\II_{\VV}$ is the identity on $\VV$.

\begin{lemma}[$\mathbb{T}$-Coercivity of $a_1(\cdot,\cdot)$]\label{lemma:T-coer}
Let $a_1(\cdot,\cdot)$ be the sesquilinear form introduced in \eqref{eq:a_1_form}.
Under the hypotheses of Assumption \ref{assum:tcoer}, $a_1(\cdot,\cdot)$ is $\mathbb{T}$-coercive.
\end{lemma}
\begin{proof}
The proof follows the approach presented in \cite{ciar12}.  First, a straightforward calculation shows that 
\begin{equation}\label{eq:prop_T}
		\mathbb{T} \ppsi_n =
		\begin{cases}
			-\ppsi_n, & 0\leq n\leq \overline{m},\\
			\ppsi_n, & n>\overline{m}.
		\end{cases}
\end{equation}
Thus, $\mathbb{T}^2 = \II$, implying that $\mathbb{T}$ is bijective. 
To prove the $\mathbb T$-coercivity, 
we follow Lemma~5 from \cite{Cristian2024} to obtain that for all $\vv \in \VV$,
\begin{eqnarray}\label{eq:t_coerc}
		a_1\(\vv,\mathbb{T}\vv\)  & = &  \sum_{0\leq n\leq \overline{m}} \alpha_n a_1\(\vv,\mathbb{T}\ppsi_n\) + \sum_{ n > \overline{m}}\alpha_n a_1(\vv,\mathbb{T}\ppsi_n) \\ 
		& = &\sum_{0\leq n\leq \overline{m}}\left(\frac{\omega^2-\kappa_n}{1+\kappa_n} \right)\,\alpha^2_n  + \sum_{ n > \overline{m}}  \left(\frac{\kappa_n-\omega^2}{1+\kappa_n} \right)\,\alpha^2_n \geq  \gamma_{\min} \,\Vert\vv\Vert^2_{1,\mu,\lambda,\rho},\nonumber
\end{eqnarray}
with 
\begin{equation}\label{eq:alpha_min}
\gamma_{\min} = \min_{n\geq 0}  \left\vert\frac{\omega^2-\kappa_n}{1+\kappa_n} \right\vert > 0.
\end{equation}
This concludes the proof. 
\end{proof}

\begin{lemma}[Continuity] \label{lem:Acont}
	There exist two constants $C_1, C_2>0$, depending on the physical coefficients of the model and geometry of $\Omega$,  such that 
	\begin{equation*}
			\Vert \mathcal{A}(\uve)\Vert_{\UU} \leq C_1\, \Vert \uve \Vert_{\UU} \quad \forall\ \uve\in\UU,
		\end{equation*}
	and
	\begin{equation*}
			\Vert \mathcal{F}  \Vert_{\UU^*}\leq C_2 \,\(\Vert\ff\Vert_0 + \Vert\gf\Vert_0 + \Vert H\Vert_0\).
	\end{equation*}
\end{lemma}
\begin{proof} Using the Cauchy--Schwarz inequality, H\"older's inequality, \eqref{eq:tensor_K}, Poincar\'e--Steklov inequality, and Korn's inequality \eqref{eq:korn} yields
\begin{eqnarray*}
	\left| \langle \mathcal{A}_{1,2}(\uve),\vve \rangle \right| 
	&\leq&\omega^2\rho\nnormZero{\uu}\nnormZero{\vv} + 2\mu\nnormZero{\varepsilon(\uu)}\nnormZero{\varepsilon(\vv)} + \lambda\nnormZero{\ddiv\uu}\nnormZero{\ddiv\vv}
    \nonumber \\ 
    && + \omega^2\rho_f\nnormZero{\ww}\nnormZero{\vv} + \omega^2\rho_f\nnormZero{\uu}\nnormZero{\zz} 
	+ \omega^2\rho_w\nnormZero{\ww}\nnormZero{\zz}\nonumber\\
    &&+ \omega k_{\min}^{-1}\nnormZero{\ww}\nnormZero{\zz}\nonumber\\
	&\leq& C\left((2\mu)^{1/2}\nnormZero{\varepsilon(\uu)}+
	(\lambda)^{{1}/{2}}\nnormZero{\ddiv\uu}+\nnormDiv{\ww}\right) \nonumber \\
	& & \times \left((2\mu)^{1/2}\nnormZero{\varepsilon(\vv)}+(\lambda)^{{1}/{2}}\nnormZero{\ddiv\vv}+\nnormDiv{\zz}\right).
\end{eqnarray*}
Applying the same tools to all other sub-operators and the right-hand side, and using Assumption \ref{model coefficients}, proves the 
statement of the lemma in a straightforward way. 
\end{proof}

\begin{lemma}[Well-posedness of $\mathcal{A}_{\mathrm{sub}}$]\label{lemma:bijective-Asub}
	Under the hypotheses of Assumption~\ref{model coefficients}, 
	provided that 
\begin{equation}\label{eq:assum_for_Asub}
    \lambda > \frac{\alpha^2}{c_0} \max \left\{\frac12, \frac8{\gamma_{\min}} \right\}, \quad
    \gamma_{\min}-\frac{2\omega^3\rho_f^2}{\rho\(k_{\max}^{-1}-\omega\rho_w\)} > 0,
\end{equation}
and the frequency is bounded from above by $\omega_{\rm crit}$ defined in \eqref{remark:low frequency}, 
	the operator $\mathcal{A}_{\mathrm{sub}}$ is bijective.
\end{lemma}

\begin{proof}   
    As already mentioned at the beginning of this section, we denote in this proof a vector from $\VV\times\ZZ\times Q$ by an arrow on top, e.g., 
    $\vve = (\vv,\zz,q)$. Now, an operator $\mathsf{R}_1\ : \ \VV\times\ZZ\times Q \to \VV\times\ZZ\times Q$ is defined by 
    $\mathsf{R}_1\uve=(\uu,\ww,-ip+\gamma\ddiv\ww)$, where $\gamma$ is a real parameter to be determined. 
    Clearly, this operator is linear and injective. The operator $\mathsf{R}_1$ is also surjective, 
    since $\uu=\uu'$, $\ww=\ww'$, $p = (\gamma\ddiv \ww'-p')/i$ is mapped to given $(\uu',\ww',p')$.
    Altogether, $\mathsf{R}_1$ is an isomorphism. Its adjoint operator is denoted by $\mathsf{R}_1^*$. To get the bijectivity of $\mathcal{A}_{\mathrm{sub}}$, we perform the following three main steps.
    
    \textbf{Step 1:} {\em Show that $\mathsf{R}^{*}_1\mathcal{A}_{\mathrm{sub}}$ satisfies a G\r{a}rding inequality.}
    A straightforward calculation yields
    \begin{eqnarray} \label{eq:garding1}
        \langle \mathcal{A}_{\mathrm{sub}}\uve, \mathsf{R}_1\uve \rangle & = & a_1(\uu,\uu)+a_2(\ww,\ww)-b_2^*(\uu,\gamma\ddiv\ww) -b_3^*(\ww,\gamma\ddiv\ww) \nonumber\\ 
        &&+d(p,-ip+\gamma\ddiv\ww) + \sum_{i=1}^3\Phi_i,
    \end{eqnarray}
    with 
    \[
    \Phi_1 = b_1(\ww,\uu)+b_1^*(\uu,\ww), \; \Phi_2 = b_2(p,\uu)-ib_2^*(\uu,p),\;
    \Phi_3 = b_3(p,\ww)-ib_3^*(\ww,p).
    \]
    Notice that these sesquilinear forms satisfy 
    \begin{eqnarray*}
    \re\Phi_1+\im \Phi_1 &=& 2\re b_1(\ww,\uu), \\ 
    \re\Phi_2+\im \Phi_2 &=& 0, \; \re\Phi_3+\im \Phi_3 = 0.
    \end{eqnarray*}
    Inserting these relations into \eqref{eq:garding1} yields
    \begin{eqnarray} \label{eq:garding2}
    \lefteqn{\re \langle\mathcal{A}_{\mathrm{sub}}\uve, \mathsf{R}_1\uve \rangle + \im\langle\mathcal{A}_{\mathrm{sub}}\uve, \mathsf{R}_1\uve \rangle }\nonumber\\
    &\geq& a_1(\uu,\uu) + 2\re b_1(\ww,\uu) + (\re+\im)a_2(\ww,\ww) + \gamma\nnormZero{\ddiv\ww}^2 + c_0\nnormZero{p}^2 \nonumber\\
    && - (\re+\im)b_2^*(\uu,\gamma\ddiv\ww) + (\re+\im)d(p, \gamma\ddiv\ww).
    \end{eqnarray}
    Utilizing the expressions from \eqref{eq:a_1_form}-\eqref{eq:d_form}, and repeatedly using the Cauchy--Schwarz and Young's inequalities gives 
    \begin{eqnarray*}
        2\re b_1(\ww,\uu) &\leq& \frac{\omega^2\rho_f}{\epsilon_1}\nnormZero{\uu}^2 + \omega^2\rho_f\epsilon_1\nnormZero{\ww}^2,\\
        (\re+\im)b_2^*(\uu,\gamma\ddiv\ww)&\leq&\frac{\gamma\alpha}{\epsilon_2}\nnormZero{\ddiv\uu}^2 + \gamma\alpha\epsilon_2\nnormZero{\ddiv\ww}^2,\\
        (\re+\im)d(p,\gamma\ddiv\ww)&\leq&\frac{c_0\gamma}{\epsilon_3}\nnormZero{p}^2 + c_0\gamma\epsilon_3\nnormZero{\ddiv\ww}^2.
    \end{eqnarray*}
    Substituting these inequalities into \eqref{eq:garding2} and using the bound of the hydraulic mobility \eqref{eq:tensor_K} leads to 
    \begin{eqnarray*}
    \lefteqn{
    (\re+\im) \langle\mathcal{A}_{\mathrm{sub}}\uve, \mathsf{R}_1\uve \rangle}\\
    &\geq & \(-\omega^2\rho-\frac{\omega^2\rho_f}{\epsilon_1}\)\nnormZero{\uu}^2 + 2\mu\nnormZero{\varepsilon(\uu)}^2 + \(\lambda-\frac{\gamma\alpha}{\epsilon_2}\)\nnormZero{\ddiv\uu}^2 \\
        && + \(\omega k_{\max}^{-1}-\omega^2\rho_w-\omega^2\rho_f\epsilon_1\)\nnormZero{\ww}^2 \\
        && + \gamma \( 1- \alpha\epsilon_2-c_0\epsilon_3\)\nnormZero{\ddiv\ww}^2 + \(c_0-\frac{c_0\gamma}{\epsilon_3}\)\nnormZero{p}^2.
    \end{eqnarray*}
    We set $\epsilon_1=\nicefrac{(\omega k_{\max}^{-1}-\omega^2\rho_w)}{2\omega^2\rho_f}$, $\epsilon_2=\nicefrac{1}{3\alpha}$, $\epsilon_3=\nicefrac{1}{3c_0}$, $\gamma=\nicefrac{\epsilon_3}{2}$. 
    Taking into account the restriction to the wave number \eqref{remark:low frequency}, one finds that the term $\omega k_{\max}^{-1}-\omega^2\rho_w$ is positive. 
    It follows, using also the first assumption from \eqref{eq:assum_for_Asub}, that all coefficients of the above terms are positive. 
    Hence, we proved a G\r{a}rding inequality, namely that there is a positive constant $C$ such that
    \begin{equation} \label{eq:garding2-2}
        2|\langle\mathsf{R}_1^*\mathcal{A}_{\mathrm{sub}}\uve, \uve \rangle| + \(\omega^2\rho+\frac{2\omega^3\rho_f^2}{ k_{\max}^{-1}-\omega\rho_w}\)\nnormZero{\uu}^2 \geq C\Vert\uu\Vert^2_{\UU}.
    \end{equation}
    
	\textbf{Step 2:} {\em $\mathsf{R}_2^*\mathcal{A}_{\mathrm{sub}}$ is injective.}
	Define a bijective operator $\mathsf{R}_2\ : \ \VV\times\ZZ\times Q \to \VV\times\ZZ\times Q$ by 
    $\mathsf{R}_2\uve=(\mathbb{T}\uu, \ww, -ip)$, with the adjoint operator $\mathsf{R}_2^*$, 
    where $\mathbb{T}$ was defined in \eqref{eq:prop_T}, then 
    \begin{eqnarray*} 
        \langle \mathcal{A}_{\mathrm{sub}}\uve,\mathsf{R}_2\uve \rangle &=& a_1(\uu,\mathbb{T}\uu) - \omega^2\rho_w\nnormZero{\ww}^2 
            +i\omega\nnormZero{\mathbf{K}^{-\frac{1}{2}}\ww}^2 + c_0 i\nnormZero{p}^2 \nonumber \\
            && + \Phi_1^{\prime} + \Phi_2^{\prime} + \Phi_3,
    \end{eqnarray*}
    with 
\[
\Phi_1^{\prime} = b_1(\ww,\mathbb{T}\uu)+b_1^*(\uu,\ww), \quad \Phi_2^{\prime} = b_2(p,\mathbb{T}\uu)-ib_2^*(\uu,p).
\]
	Let $\uve\in\UU$ such that $\mathsf{R}^{*}_2\mathcal{A}_{\mathrm{sub}}\uve=\zero$. Since 
    both the real and imaginary parts of $\mathsf{R}^{*}_2\mathcal{A}_{\mathrm{sub}}\uve$ vanish and 
    $\re\Phi_3+\im\Phi_3=0$, we get
	\begin{eqnarray}\label{eq:T1_inj}
		0&= & |\langle \mathcal{A}_{\mathrm{sub}}\uve,\mathsf{R}_2\uve \rangle| = 
		\re\langle \mathcal{A}_{\mathrm{sub}}\uve,\mathsf{R}_2\uve \rangle + \im\langle         
            \mathcal{A}_{\mathrm{sub}}\uve,\mathsf{R}_2\uve \rangle \nonumber \\
		&\geq& a_1(\uu,\mathbb{T}\uu) + \(\omega k^{-1}_{\max}-\omega^2\rho_w\)\nnormZero{\ww}^2 
            + c_0\nnormZero{p}^2 + \re\Phi_1^{\prime}+\im\Phi_1^{\prime} \nonumber\\
            &&+ \re\Phi_2^{\prime}+\im\Phi_2^{\prime}.
	\end{eqnarray}
    Expanding $\uu$ in terms of the eigenfunctions \eqref{eq:expansion_efct} and utilizing the property \eqref{eq:prop_T}
	of the operator $\mathbb{T}$, Young's inequality, and the $L^2$ orthogonality  of the eigenfunctions \eqref{eq:L2_ortho} yields
	\begin{eqnarray}\label{eq:T1_inj_0}
		\re\Phi_1^{\prime}+\im\Phi_1^{\prime}&=&-\omega^2\rho_f\re(\ww, \uu+\mathbb{T}\uu) -\omega^2\rho_f\im(\ww, \mathbb{T}\uu-\uu)\nonumber \\
        &=&-\omega^2\rho_f\re\(\ww, \sum_{n\geq\overline{m}}2\alpha_n\ppsi_n\)+\omega^2\rho_f\im\(\ww,\sum_{0\leq n\leq\overline{m}}2\alpha_n\ppsi_n\)\nonumber \\
		&\geq& -\omega^2\rho_f\epsilon_1\nnormZero{\ww}^2 - \frac{\omega^2\rho_f}{\epsilon_1}\nnormZero{\uu}^2.
	\end{eqnarray}
	\begin{eqnarray}\label{eq:T1_inj_1}
		\re\Phi_2^{\prime}+\im\Phi_2^{\prime}&=&\alpha\re\(p,\ddiv(\uu-\mathbb{T}\uu)\) + \alpha\im\(p,\ddiv(\uu-\mathbb{T}\uu)\)\nonumber \\
		&\geq& -\alpha\epsilon_2\nnormZero{p}^2 - \frac{\alpha}{\epsilon_2}\nnormZero{\ddiv(\uu-\mathbb{T}\uu)}^2.
	\end{eqnarray}
    Utilizing  the orthogonality of the eigenfunctions with respect to the inner product $(\cdot,\cdot)_{1,\mu,\lambda}$, one has
    \begin{equation}\label{eq:T1_inj_2}
        \nnormZero{\ddiv(\uu-\mathbb{T}\uu)}^2 \leq \frac{1}{\lambda}\Vert \uu-\mathbb{T}\uu \Vert_{1,\mu,\lambda}^2 \leq \frac{4}{\lambda}\Vert\uu\Vert_{1,\mu,\lambda}^2 
		= \frac{4}{\lambda}\Vert\uu\Vert^2_{1,\mu} + 4\nnormZero{\ddiv\uu}^2.
    \end{equation}
	Using in addition the $\mathbb{T}$-coercivity \eqref{eq:t_coerc} of $a_1(\cdot,\cdot)$ yields
	\begin{eqnarray*}
		|\langle \mathcal{A}_{\mathrm{sub}}\uve, \mathsf{R}_2\uve \rangle|&\geq& \(\gamma_{\min}-\frac{4\alpha}{\lambda\epsilon_2}\) \Vert\uu\Vert_{1,\mu}^2 + 
        \(\gamma_{\min}\rho-\frac{\omega^2\rho_f}{\epsilon_1}\)\nnormZero{\uu}^2\nonumber\\
        && \(\gamma_{\min}\lambda-\frac{4\alpha}{\epsilon_2}\)\nnormZero{\ddiv\uu}^2
        + \(\omega k^{-1}_{\max}-\omega^2\rho_w-\omega^2\rho_f\epsilon_1\)\nnormZero{\ww}^2\nonumber\\
        &&+ (c_0-\alpha\epsilon_2)\nnormZero{p}^2.
	\end{eqnarray*}
    Setting $\epsilon_2=\nicefrac{c_0}{2\alpha}$ gives, together with assumption \eqref{eq:assum_for_Asub}, that the coefficients 
    in front of $\Vert\uu\Vert_{1,\mu}^2$, $\nnormZero{\ddiv\uu}^2$, and $\nnormZero{p}^2$ are  positive. 
    Like in the first step of the proof, we set $\epsilon_1=\nicefrac{(\omega k_{\max}^{-1}-\omega^2\rho_w)}{2\omega^2\rho_f}$, which is positive
    due to \eqref{remark:low frequency} and \eqref{eq:assum_for_Asub}.
    Substituting then $\epsilon_1$ into the coefficients
	of $\nnormZero{\uu}^2$ shows, together with the third constraint from \eqref{eq:assum_for_Asub}, that also the coefficient 
    in front of $\nnormZero{\uu}^2$ is positive. With \eqref{eq:T1_inj}, 
     the injectivity of $\mathsf{R}_2^{*}\mathcal{A}_{\mathrm{sub}}$ can be inferred. 
     
\textbf{Step 3:} {\em $\mathcal{A}_{\mathrm{sub}}$ is bijective.}
    Since $\mathsf{R}_1$ and $\mathsf{R}_2$ are bijective, linear, and bounded, their adjoint operators  $\mathsf{R}_1^{*}$ and $\mathsf{R}_2^{*}$ are bijective. 
    From the injectivity of $\mathsf{R}_2^*\mathcal{A}_{\mathrm{sub}}$ proved in Step~2, one can deduce the injectivity of $\mathsf{R}_1^*(\mathsf{R}_2^{*-1}\mathsf{R}_2^*)\mathcal{A}_{\mathrm{sub}}$, i.e., the injectivity of $\mathsf{R}_1^*\mathcal{A}_{\mathrm{sub}}$.
    Define the operator $K:\VV\to\VV$ by
    \begin{equation*}
        \langle K\uve,\vve\rangle = -\(\frac{\omega^2\rho}{2}+\frac{\omega^3\rho_f^2}{ k_{\max}^{-1}-\omega\rho_w}\)(\uu,\vv), 
        \quad \forall\ \uu,\vv\in\VV.
    \end{equation*}
    Note that $K$ is bounded, self-adjoint, and compact, see  $\S$8.2 in \cite{SayasVariational}. 
    Utilizing the G\r{a}rding inequality \eqref{eq:garding2-2}, the operator $\mathcal{A}_{G}:=\mathsf{R}^{*}_1\mathcal{A}_{\mathrm{sub}} - K$ is invertible.
    Since  $\mathsf{R}_1^{*}\mathcal{A}_{\mathrm{sub}}$ was shown to be injective, 
    employing the Fredholm's alternative Lemma~\ref{the-Fredholm-alternative} gives the bijectivity of $\mathsf{R}_1^{*}\mathcal{A}_{\mathrm{sub}}$, and in consequence the bijectivity of $\mathcal{A}_{\mathrm{sub}}$.
\end{proof}

\begin{remark}[On the assumptions for Lemma~\ref{lemma:bijective-Asub}]
The operator $\mathcal{A}_{\mathrm{sub}}$ can be written in the block form \eqref{eq:oper-decomp}. 
Standard theory for this type of linear block systems states that $\mathcal{A}_{\mathrm{sub}}$ is bijective if and only if the 
Schur complement $\mathcal{B}^*\mathcal{A}_{1,2}^{-1}\mathcal{B}+D $ is bijective. If, e.g., all eigenvalues of 
$\mathcal{A}_{1,2}$ are positive and all eigenvalues of $D$ nonnegative, then the well-posedness of the Schur 
complement operator is given, since then the eigenvalues of the Schur complement are positive. However, the location of the eigenvalues of $\mathcal{A}_{1,2}$ is not known and a reasoning as above cannot be applied. 
In fact, one can easily construct (finite-dimensional) examples where the 
Schur complement, and consequently the corresponding block system, is not invertible although $\mathcal{A}_{1,2}$ and $D$ are invertible.
Thus, it is not unexpected that the well-posedness of $\mathcal{A}_{\mathrm{sub}}$ holds only with appropriate assumptions.

For the studied problem, all eigenvalues of $D$ are in the right complex plane. Thus, one can expect to prove the 
well-posedness of the problem if the second term of the Schur complement dominates the first one. Roughly speaking, this situation is
given, for fixed $\omega$, if $c_0$ is large, $\alpha$ is small, and $k_{\mathrm{max}}$ is small. These requirements can be found 
in the assumption that \eqref{remark:low frequency} is satisfied and in the first condition of \eqref{eq:assum_for_Asub}.
\begin{itemize}
\item Concerning the first condition in \eqref{eq:assum_for_Asub}, it is known from \cite{LeeParameter} that $c_0$ scales like $\alpha^2/\lambda$. Hence, this assumption is reasonable. 
\item The second condition in \eqref{eq:assum_for_Asub} contains the constant $\gamma_{\min}$ defined in 
\eqref{eq:alpha_min}. Hence, this condition will be satisfied if $\omega$ is sufficiently small.
\end{itemize}   
Altogether, the conditions in \eqref{eq:assum_for_Asub} restrict the analysis to the low frequency range. 
\end{remark}

\begin{lemma}[Compactness of $\mathcal{C}_1$ and $\mathcal{C}_2$]\label{lemma:compact-C}
	The operators $\mathcal{C}_1$ and $\mathcal{C}_2$, defined in \eqref{eq:op-C1} and \eqref{eq:op-C2}, are compact.
\end{lemma}
\begin{proof}
    Let $\mathfrak{C}_1: S\rightarrow \VV^*$ and $\mathfrak{C}_2: S\rightarrow Q$ be the operators induced by 
    the sesquilinear forms $c_1(\cdot,\cdot)$ and $c_2(\cdot,\cdot)$, i.e.,
    \begin{equation*} 
        \langle \mathfrak{C}_1(T), \vv \rangle=c_1(T,\vv), \quad \langle \mathfrak{C}_2(T), q \rangle= c_2(T,q).
    \end{equation*}
    Moreover, let $I$ be the identity operator and $i_c$ be the compact embedding from $H^1(\OO)$ into $L^2(\OO)$.
    Define the operator $\mathfrak{C}_{\ddiv}: Q\rightarrow \VV^*$, such that
    \begin{equation*}
        \langle \mathfrak{C}_{\ddiv} q, \vv \rangle = (q, \ddiv\vv), \quad \forall\ q\in Q,\, \vv\in \VV.
    \end{equation*}
    Due to the property that $\nnormZero{\ddiv\vv}\leq\sqrt{d}\nnormZero{\nabla\vv}$, $\mathfrak{C}_{\ddiv}$ is a bounded and linear operator.  
    Then, it is straightforward that $\mathfrak{C}_1=-\beta I\circ\mathfrak{C}_{\ddiv}\circ i_c$, as a composition of a compact operator and bounded linear operators, is compact. 
    Furthermore, $\mathfrak{C}_2=-b_0 I \circ i_c$ is also compact. 
    Using the expressions   \eqref{eq:op-C1} and \eqref{eq:op-C2} of $\mathcal{C}_1$ and $\mathcal{C}_2$, respectively, completes the proof. 
\end{proof}

\begin{lemma}[Injectivity of $\mathcal{A}$] \label{lemma:InjecA}
	Let Assumption~\ref{model coefficients} be satisfied and let $\omega$ be smaller than 
    $\omega_{\rm crit}$ defined in \eqref{remark:low frequency}. 
    Assume that the second condition from \eqref{eq:assum_for_Asub} is satisfied and 
    \begin{equation} \label{eq:assum_for_AInj}
       \tau<\frac{1}{\omega}, \quad
        a_0 > \frac{2b_0^2}{c_0}, \quad 
        \lambda > \frac{8}{\gamma_{\min}}\max\left\{\frac{2\alpha^2}{c_0}, \frac{\beta^2}{a_0-\nicefrac{2b_0^2}{c_0}} \right\}.
    \end{equation}
    Then, the operator $\mathcal{A}$ is injective.
\end{lemma}
\begin{proof}
	Define the operator $\mathsf{R}\uve\ : \ \UU\to \UU$, $(\uu,\ww,p, T) \to (\mathbb{T}\uu, \ww, -ip, T)$, with the adjoint operator $\mathsf{R}^*$.
    Since $\mathbb{T}$ is bijective, see the proof of Lemma~\ref{lemma:T-coer}, $\mathsf{R}$ is also bijective. 
    Consider
	\[
		\langle \mathcal{A}\uve,\mathsf{R}\uve \rangle = \langle \mathcal{A}_{\mathrm{sub}}\uve,\mathsf{R}_2\uve \rangle 
		+ ia_0\nnormZero{T}^2 + \frac{1-\omega\tau i}{\omega+\omega^3\tau^2}\nnormZero{\mathbf{\Theta}^{\frac{1}{2}}\nabla T}^2 + \Phi_4 + \Phi_5,
	\]
        where $\mathsf{R}_2$ is defined in Step~2 of the proof of Lemma~\ref{lemma:bijective-Asub}, and 
        \[
        \Phi_4=c_1(T,\mathbb{T}\uu)-ic_1^*(\uu,T), \quad 
        \Phi_5=ic_2(T,p)+ic_2^*(p,T).
        \]
	To prove the injectivity of $\mathcal{A}$, let $\uve\in \UU$ be such that $\mathcal{A}(\uve)=\zero$. The goal is to conclude $\uve=\zero$. 
	If $\langle\mathcal{A}(\uve), \mathsf{R}\uve\rangle=0$, then both the real and imaginary parts of this expression are zero, so that 
    we get, also using \eqref{eq:tensor_theta}, 
	\begin{eqnarray} \label{eq:garding4}
		0&=&\re\langle\mathcal{A}(\uve), \mathsf{R}\uve\rangle + \im\langle\mathcal{A}(\uve), \mathsf{R}\uve\rangle  
            \nonumber\\
		&\geq& \re\langle \mathcal{A}_{\mathrm{sub}}\uve,\mathsf{R}_2\uve \rangle + \im\langle\mathcal{A}_{\mathrm{sub}}\uve,\mathsf{R}_2\uve \rangle + a_0\nnormZero{T}^2 
             + \frac{\theta_{\min}(1-\omega\tau)}{\omega+\omega^3\tau^2}\nnormZero{\nabla T}^2  \nonumber \\
            && + \re\Phi_4 + \im\Phi_4 + \re\Phi_5 + \im\Phi_5.
	\end{eqnarray}
	Using \eqref{eq:prop_T}, \eqref{eq:T1_inj_2}, the Cauchy--Schwarz and Young's inequality, we derive
	\begin{eqnarray}\label{eq:lem_A_inj_0}
		\re\Phi_4 + \im\Phi_4
		&=& \beta\re(T,\ddiv(\uu-\mathbb{T}\uu))+\beta\im(T,\ddiv(\uu-\mathbb{T}\uu))\nonumber \\
        &\geq& -\beta\epsilon_3\nnormZero{T}^2 - \frac{\beta}{\epsilon_3}\nnormZero{\ddiv(\uu-\mathbb{T}\uu)}^2 \nonumber \\
        &\geq& -\beta\epsilon_3\nnormZero{T}^2 - \frac{4\beta}{\lambda\epsilon_3}\Vert\uu\Vert_{1,\mu}^2 - \frac{4\beta}{\epsilon_3}\nnormZero{\ddiv\uu}^2 
	\end{eqnarray}
    and similarly,
	\begin{equation*}
		\re\Phi_5 + \im\Phi_5=-2b_0\re(T,p)\geq -b_0\epsilon_4\nnormZero{T}^2 - \frac{b_0}{\epsilon_4}\nnormZero{p}^2.
	\end{equation*}
    Notice that \eqref{eq:T1_inj}--\eqref{eq:T1_inj_2} provides a lower bound for the first two terms on the right-hand side of \eqref{eq:garding4}.
	Combining this bound and the two estimates from above yields
	\begin{eqnarray}\label{eq:lem_A_inj_1}
		|\langle \mathcal{A}\uve,\mathsf{R}\uve \rangle|&\geq& \(\gamma_{\min}-\frac{4\alpha}{\lambda\epsilon_2}-\frac{4\beta}{\lambda\epsilon_3}\) \Vert\uu\Vert_{1,\mu}^2 + 
            \(\gamma_{\min}\rho-\frac{\omega^2\rho_f}{\epsilon_1}\)\nnormZero{\uu}^2 
		\nonumber\\
		&&+ \(\gamma_{\min}\lambda-\frac{4\alpha}{\epsilon_2}-\frac{4\beta}{\epsilon_3}\)\nnormZero{\ddiv\uu}^2 
		+ (\omega k^{-1}_{\max}-\omega^2\rho_w-\omega^2\rho_f\epsilon_1)\nnormZero{\ww}^2 \nonumber\\
		&& + \(c_0-\alpha\epsilon_2-\frac{b_0}{\epsilon_4}\)\nnormZero{p}^2
		+ (a_0-\beta\epsilon_3-b_0\epsilon_4)\nnormZero{T}^2\nonumber\\
		&& +\frac{\theta_{\min}(1-\omega\tau)}{\omega+\omega^3\tau^2}\nnormZero{\nabla T}^2.
	\end{eqnarray}
Setting $\epsilon_1=\nicefrac{(\omega k_{\max}^{-1}-\omega^2\rho_w)}{2\omega^2\rho_f}$, $\epsilon_2=\nicefrac{8\alpha}{\gamma_{\min}\lambda}$, $\epsilon_3=\nicefrac{8\beta}{\gamma_{\min}\lambda}$, and $\epsilon_4=\frac{2b_0}{c_0}$, 
leads under the assumptions of the lemma to positive terms for $\uu$, $\ww$, and $\nabla T$.  
Hence $\uu=\ww=\zero$ and, using Poincar\'e's inequality, $T=0$.  Inserting these values in \eqref{eq:garding4} we conclude that $p=0$, 
due to $c_0>0$ in Assumption~\ref{model coefficients}.
The injectivity of $\mathsf{R}^*\mathcal{A}$ follows. Due to the bijectivity of $\mathsf{R}^*$, we obtain the injectivity of $\mathcal{A}$.
\end{proof}

\begin{remark}[Size of $\tau$, incompressible materials]
        Introducing the Maxwell--Vernotte--Cattaneo relaxation term with $\tau$ is to avoid infinite speed associated with heat conduction, 
    and to predict the slow thermal (T) wave \cite{JCPBonettiNumerical}.
    Some works, e.g.,  on quasi-static thermo-poroelasticity \cite{AntoniettiDiscontinuous,Yi2024Physics},
    do not include the relaxation term in their models, thus they consider the case $\tau=0$. Hence, assuming 
    that $\tau$ is bounded, as in the first condition of \eqref{eq:assum_for_AInj}, seems to be not restrictive. 

    Consider the incompressible case $\lambda=\infty$, which is equivalent to $\ddiv\uu=0$. Then, the left-hand sides of \eqref{eq:T1_inj_1} and \eqref{eq:lem_A_inj_0} vanish, so that the terms with $\epsilon_2$ and $\epsilon_3$ do not appear in \eqref{eq:lem_A_inj_1}. 
    We take the same values for $\epsilon_1$ and $\epsilon_4$. 
    Applying basically the same arguments as in the compressible case, we obtain the injectivity of $\mathcal{A}$.
    Among the conditions in \eqref{eq:assum_for_AInj}, only the first one needs to be retained. 
\end{remark}

\begin{theorem}[Well-posedness of problem \eqref{weakF}] \label{Th:ContiWellPosed}
Let the assumptions of Lemmas~\ref{lemma:bijective-Asub} and~\ref{lemma:InjecA} hold. 
Then, for every $\mathcal F \in \UU^*$, problem \eqref{weakF} has a unique solution $\uve \in \UU$. There exists a positive constant $C$ such that $\Vert\uve\Vert_{\UU}\leq C\Vert\mathcal{F}\Vert_{\UU^*}$.
\end{theorem}
\begin{proof}
Notice that 
\[
\left| a_3(T,T)\right| \ge
\left|\re a_3(T,T)\right|\geq \frac{\theta_{\min}}{\omega+\omega^3\tau^2}\nnormZero{\nabla T}^2.
\]
Thus, the operator $\mathcal{A}_3$ is bijective by the Lax--Milgram lemma for convex spaces, e.g., see Proposition~5.4 in \cite{SayasVariational}.
From \eqref{eq:oper-decomp}, we can rewrite the injective operator $\mathcal{A}$ (Lemma~\ref{lemma:InjecA}) as a bijective operator adding a compact one:
	\begin{equation*}
		\mathcal{A} = 
		\begin{pmatrix}
			\mathcal{A}_{\mathrm{sub}} & \mathcal{C}_1  \\ \mathcal{C}_2 & \mathcal{A}_3
		\end{pmatrix}
		=
		\begin{pmatrix}
			\mathcal{A}_{\mathrm{sub}} &   \\  & \mathcal{A}_3
		\end{pmatrix}
		+
		\begin{pmatrix}
			& \mathcal{C}_1  \\ \mathcal{C}_2 & 
		\end{pmatrix}
		.\,
	\end{equation*}
The existence and uniqueness of a solution follow by applying Lemmas~\ref{lemma:bijective-Asub}, \ref{lemma:compact-C}, \ref{lemma:InjecA}, and the Fredholm alternative from Lemma~\ref{the-Fredholm-alternative}. The boundedness follows also from the Babu$\check{s}$ka--Lax--Milgram lemma.
\end{proof}

	\section{Analysis of the Discrete Problem} \label{sec:discrete-analysis}

\subsection{\texorpdfstring{$\boldsymbol{BR}$-$\boldsymbol{RT}_0$-$P_0$-$P_1$}{} finite element formulation}

Let $\{\Trih\}_{h>0}$ denote a shape regular family of triangulations of $\overline{\OO}$. 
The set of facets is denoted by $\mathcal{F}_h=\mathcal{F}_h^{o}\cup\mathcal{F}_h^{\partial}$, 
where $\mathcal{F}_h^{o}$ denotes the interior facets and $\mathcal{F}_h^{\partial}$ denotes the facets on the boundary.
In addition, $\mathcal{F}_h^{\Gamma_t}$ is the set of facets on the boundary $\Gamma_t$ 
and $\mathcal{F}_h^{o,t}=\mathcal{F}_h^{o}\cup\mathcal{F}_h^{\Gamma_t}$.
For a mesh cell $K\in\mathcal T_h$ we denote by $h_K$ its diameter and introduce
$h:=\max \{h_K:K\in \Trih\}$ as the characteristic mesh size.
For the interior facet $e\in\mathcal{F}_h^{o}$, where $e$ is the common facet of $K^{+}$ and $K^{-}$, $K^{\pm}\in\Trih$, 
let $\boldsymbol{n}_{e}=\boldsymbol{n}_{e, K^{+}}=-\boldsymbol{n}_{e, K^{-}}$ be the unit normal vector on $e$. 
The orientation of this vector can be chosen arbitrarily, but fixed. 
For the boundary facet $e\in\mathcal{F}_h^{\partial}$, we set $\boldsymbol{n}_{e}=\boldsymbol{n}_{e,K}$.

The finite element space for the displacement is the Bernardi--Raugel ($\boldsymbol{BR}$) element, which is developed for the Stokes model, e.g., see 
\cite{BRStokes}, Chapter~II.2.2 in \cite{Raviart0Finite}, or Remark~3.138 in \cite{Joh16}. 
Define the continuous piecewise linear polynomial space
\[
	\VV_{h,l}  :=\left\{\vv\in \VV: \vv|_K\in [P_1(K)]^d,  \forall\ K\in\Trih\right\},
\]
where $P_j(K)$, $j\geq 0$, denotes the space of polynomials of degree no more than $j$ on a mesh cell $K$. 
Further, with every facet $e\in\mathcal{F}_h^{o,t}$, we associate a vector-valued function $\boldsymbol{\Psi}_e$, 
where $\boldsymbol{\Psi}_e=\psi_e\boldsymbol{n}_e$ and for an interior facet $\psi_e|_{K^{\pm}}=\psi_{e,K^{\pm}}=\prod^{d+1}_{l=1,l\neq j^{\pm}}\lambda_{l,K^{\pm}}$, 
Here, $\lambda_{l,K^{\pm}}$, $l=1,\dots,d+1$, are barycentric coordinates on $K^{\pm}$ and $j^{\pm}$ is the vertex opposite to the 
facet $e$ in $K^{\pm}$.
For a boundary facet, we can perform a similar construction. 
Finally, the $\boldsymbol{BR}$ element can be written as the classical linear finite element space enriched by facet bubble functions:
\begin{equation*}
    \VV_h=\VV_{h,l} \oplus \VV_b, \quad \VV_b=\Span\left\{\boldsymbol{\Psi}_e\right\}_{e\in\mathcal{F}_h^{o,t}}.
\end{equation*}
Define the canonical interpolation operator $\Pi_h^v: \VV\cap[C(\Bar{\Omega})]^d\to\VV_h$, 
$\Pi_h^v\vv=\Pi_{h,l}^v\vv + \sum_{e\in\mathcal{F}_h^{o,t}}v_e\boldsymbol{\Psi}_e$, 
where $\Pi^v_{h,l}: \VV\cap[C(\Bar{\Omega})]^d\to\VV_{h,l}$ is the standard Lagrangian interpolation
and 
\[
v_e=\frac{\int_{e}(\boldsymbol{v}-\Pi_{h,l}\boldsymbol{v})\cdot\boldsymbol{n}_e\ \mathrm{d}\bs}{\int_{e}\prod^{d+1}_{l=1,l\neq j^{\pm}}\lambda_l\ \mathrm{d}\bs}.
\]
The interpolation $\Pi_h^v$, along with subsequent projections, operates separately on the real and imaginary parts.
Consequently, these operators inherit the properties established for real-valued functions.
The operator $\Pi_h^v$ satisfies the following properties \cite{BRStokes,Yi2017Study}:
\begin{eqnarray}
    \int_{\Omega}\nabla\cdot(\vv-\Pi_h^v\vv)\ \mathrm{d}\bx &=& 0, \nonumber \\ 
	\Vert \vv-\Pi_h^v\vv\Vert_{s} &\leq& Ch_k^{m-s}\Vert\vv\Vert_{m}, \quad 0\leq s\leq 1, 1\leq m\leq 2. \label{prop:Piv2}
\end{eqnarray}

    The lowest-order Raviart--Thomas ($\boldsymbol{RT}_0$) finite element space \cite{boffi2013mixed} for the filtration displacement is denoted by
\[
    \ZZ_h:=\left\{\zz\in \ZZ:\zz|_K \in\left[P_{0}(K)\right]^{d} \oplus \vecb{x} P_{0}(K), \forall\ K\in\mathcal{T}_{h} \right\}.
\]
The space of piecewise constants ($P_0$) for the pore pressure reads
\[
    Q_{h}:=\left\{q\in Q:\left.q\right|_{K} \in P_{0}(K), \forall\ K\in\mathcal{T}_{h}\right\}.
\]
The Raviart--Thomas interpolation $\Pi_h^z:\ZZ\to\ZZ_h$ 
and the orthogonal $L^2$ projection $P_h:Q\to Q_h$ are defined by
\begin{eqnarray}
    ((\zz-\Pi_h^z\zz)\cdot\vecb{n}_e,1)_e&=&0, \quad\forall\ e\in\mathcal{F}_h, \nonumber \\
    (q-P_hq,r_h)&=&0, \quad\forall\ r_h\in Q_h, \label{prop:Ph1}
\end{eqnarray}
respectively. The following approximation and commutative properties can be found in \cite{raviart2006mixed,Yi2017Study}:
\begin{eqnarray}
   \Vert\zz-\Pi_h^z\zz\Vert_{0,K} &\leq& Ch_K\vert\zz\vert_{1,K}, \quad\forall\ K\in\mathcal{T}_h, \label{prop:Piz2}\\
   \Vert\ddiv(\zz-\Pi_h^z\zz)\Vert_{0,K} &\leq& Ch_K\vert\ddiv\zz\vert_{1,K}, \quad\forall\ K\in\mathcal{T}_h,\label{prop:Piz3}\\
   \Vert q-P_hq\Vert_{0,K} &\leq& C{h_K^s}\vert r\vert_{s,K}, \quad\forall\ K\in\mathcal{T}_h, \, s=0,1,\label{prop:Ph2}\\
   P_h\ddiv\Pi_h^v\vv &=& P_h\ddiv\vv, \quad\forall\ \vv\in\VV\cap [C(\overline{\Omega})]^d, \label{prop:comm1}\\
   \ddiv\Pi_h^z\zz &=& P_h\ddiv\zz, \quad \forall\ \zz\in\ZZ. \nonumber 
\end{eqnarray}

Finally, the piecewise linear polynomial space ($P_1$) for the temperature is defined as
\begin{align*}
	S_h := \left\{s\in S: s|_K\in P_1(K), \forall\ K\in\Trih \right\}.
\end{align*}
The Lagrangian (nodal) interpolation is denoted by $\Pi_h^s:\vecb{C}(\overline{\Omega})\to S_h$ and we have the standard approximation property, 
e.g., see Theorem~4.4.20 in \cite{brenner2008mathematical}, 
\begin{equation}
    \left\Vert s-\Pi_h^s s\right\Vert_K+h_K\left\vert s-\Pi_h^s s\right\vert_{1,K}\leq Ch_K^2\vert s\vert_{2,K}, \label{prop:Pis}
    \quad\forall\ K\in\mathcal{T}_h.
\end{equation}

As it will be motivated after having introduced the discrete problem, concretely in the proof of Theorem~\ref{Th:DiscWellPosed}, we 
define new sesquilinear forms for the finite element formulation
\begin{alignat}{4}
		&a_{h1}: \VV_h\times\VV_h\to\mathbb{C}, \ && a_{h1}(\uu_h,\vv_h) &&:= &&\ -\omega^2\rho\uprod{\uu_h}{\vv_h} 
		+ 2\mu\uprod{\varepsilon(\uu_h)}{\varepsilon(\vv_h)} \label{eq:a_h1_form} \\
		& && && &&  \   + \lambda\uprod{P_h\ddiv\uu_h}{P_h\ddiv\vv_h}, \nonumber \\
		&c_{h1}: S_h\times\VV_h\to\mathbb{C}, \ &&c_{h1}(T_h,\vv_h) &&:= &&\  -\beta\uprod{T_h}{P_h\ddiv\vv_h}. \nonumber 
\end{alignat}
In \eqref{weakF} we replace $a_1(\cdot,\cdot)$ and $c_1(\cdot,\cdot)$ by $a_{h1}(\cdot,\cdot)$ and $c_{h1}(\cdot,\cdot)$, respectively, to define the discrete operator $\mathcal{A}_h$. 

If $a_0$ and $b_0$ are very small, or even $a_0=b_0=0$, and $\theta_{\mathrm{min}}$ is (very) small, too, then 
the discrete problem is close to be of saddle point type. Since the finite element spaces for the solid displacement and temperature do
not satisfy a discrete inf-sup condition, there might be the onset of an instability, leading to temperature oscillations, compare the numerical studies in Section~\ref{sec:ana_sol}. To mitigate this issue, we introduce an additional stabilization term, in the spirit of  \cite{brezzi1984stabilization}, of
the form  $\mathcal{S}_h:\UU_h\rightarrow\UU_h^*$ as follows:
\begin{equation}\label{eq:stabilization}
    \langle \mathcal{S}_h(\uve_h),\vve_h \rangle := \delta\sum_{K\in\mathcal{T}_h}h_K^2(\nabla T_h, \nabla s_h)_K,
\end{equation}
where $\delta\geq 0$ is a stabilization parameter. In situations where this term is not needed, one 
can set $\delta=0$.

Define $\mathcal{A}_{h,\mathrm{stab}}:=\mathcal{A}_h + \mathcal{S}_h$.
Then, the discrete formulation is to find $\uve_h = (\uu_h,\ww_h,p_h,T_h)\in\UU_h$ such that
\begin{equation}\label{discF}
	\left\langle \mathcal{A}_{h,\mathrm{stab}}(\uve_h),\vve_h \right\rangle = \langle \mathcal{F},\vve_h \rangle \quad \forall\ \vve_h \in \UU_h.
\end{equation}

In \eqref{eq:a_h1_form}, we implement the technique of reduced integration \cite{Yi2017Study} 
to obtain the locking-free property with respect to $\lambda$. 
Specifically, this is due to the approximation property \eqref{prop:Ph2}, the commutativity \eqref{prop:comm1}, 
and the extra regularity of $\lambda\vert\ddiv\uu\vert$, which is discussed after the proof of Theorem~\ref{Th:errorEstimate} and assumed to be satisfied.
The same approach can be found in \cite{li2023parameter,Yi2017Study}.
Considering the construction of the $\boldsymbol{BR}$ element, this effect is limited to the edge bubble basis functions, 
since the divergence of the other (linear) basis functions is inherently piecewise constant. 
In order to ensure the well-posedness of the discrete formulation \eqref{discF}, 
the same operator $P_h$ is incorporated into the sesquilinear form $c_{h1}(\cdot,\cdot)$.

 \begin{remark}[On the choice of finite element spaces]
    As $\lambda\rightarrow\infty$, the solid displacement enters a divergence-free state and the finite element space is overconstrained. 
    The research on divergence-free finite element methods  for incompressible flow problems is an active field of research.
    Among these, classical $H^1$-conforming elements, such as the Scott--Vogelius element, satisfy the discrete inf-sup condition only 
    for certain meshes and when the polynomial degree is sufficiently large.
    Here, we employ the $\boldsymbol{BR}$ element, include the $L^2$ projection $P_h$ in $a_{h1}(\cdot,\cdot)$ and $c_{h1}(\cdot,\cdot)$, 
    in order to avoid volumetric locking, and will prove the well-posedness following the same arguments as in the continuous case.
    Considering only the spaces of the solid displacement, the filtration displacement, and pore pressure, 
    we assert that the triple ($\VV_h$, $\ZZ_h$, $Q_h$) is Stokes--Biot stable \cite{Hu2018New}, which can avoid pressure oscillations.
\end{remark}

\subsection{Well-posedness}
The product space $\UU_h:=\VV_h\times \ZZ_h\times Q_h\times S_h$ is equipped with the norm
\begin{eqnarray}\label{norm:disc}
    \Vert (\vv_h, \zz_h, q_h, s_h) \Vert^2_{\UU_h} &:= 	&	
	2\mu\nnormZero{\varepsilon(\vv_h)}^2 + \lambda\nnormZero{P_h\ddiv\vv_h}^2 + \Vert\zz_h\Vert^2_{\ddiv} 
	+\nnormZero{q_h}^2 \nonumber \\
    && + \nnormOne{s_h}^2 + \delta\sum_{K\in\mathcal{T}_h}h_K^2\Vert\nabla s_h\Vert^2_{0,K}.
\end{eqnarray}
Notice that this norm can be extended to functions from $\UU$.
The continuity of $\mathcal{A}_{h,\mathrm{stab}}$ is proved along the lines of the proof of Lemma~\ref{lem:Acont}. There exists a positive constant $C$, depending on the physical and geometrical coefficients of the problem, such that 
\begin{equation*}
	\Vert \mathcal{A}_{h,\mathrm{stab}}(\uve_h)\Vert_{\UU_h} \leq C\,\Vert \uve_h \Vert_{\UU_h}.
\end{equation*}

In case no ambiguity occurs, we use the same notations as in Section \ref{sec:continuous-analysis}.
Define a new inner product in $\VV_h$ as follows:
\begin{eqnarray*}
    (\uu_h,\vv_h)_{1,\mu,\Bar{\lambda}} := (\uu_h,\vv_h)_{1,\mu} + \lambda(P_h\ddiv\uu_h,P_h\ddiv\vv_h),
\end{eqnarray*}
and denote the associated norm by $\Vert\vv_h\Vert_{1,\mu,\Bar{\lambda}}$. 
There exists a family $\{\ppsi_n^h,\kappa_n^h\}_n \in \VV_h\times\RR^+$ such that $\ppsi_n^h\neq\mathbf{0}$ and 
\begin{equation*}
	\(\ppsi_n^h,\vv^h\)_{1,\mu,\Bar{\lambda}} = \kappa_n^h\(\ppsi_n^h,\vv^h\)_{0,\rho} \quad \forall\ \vv^h\in\VV_h.
\end{equation*}
This eigenvalue problem is also well posed due to Korn's inequality.
Now, similarly to the continuous case, we can fix the corresponding $\overline{m}$ 
and define the subspace $\VV^-_h := \Span_{0\leq n\leq \overline{m}} \{\ppsi^h_n\}$.
Define $\mathbb{T}_h:=\II_{\VV_h}-2\Pihm$ of $\mathcal{L}(\VV_h)$, where $\Pihm$ is the orthogonal projection on $\VV^-_h$.
For the sake of brevity, we omit the detailed analysis of the properties of $\mathbb{T}_h$. But analogous results can be proven for this operator like in the continuous case as stated in 
\eqref{eq:prop_T}-\eqref{eq:t_coerc}. In particular, the sesquilinear form $a_{h1}(\cdot,\cdot)$ is $\mathbb{T}_h$-coercive.

\begin{theorem}[Well-posedness of \eqref{discF}] \label{Th:DiscWellPosed}
    Under the condition of Lemma \ref{lemma:InjecA}, the problem \eqref{discF} has a unique solution $\uve_h \in \UU_h$. There exists a positive constant $C$ such that $\Vert\uve_h\Vert_{\UU_h}\leq C\Vert\mathcal{F}\Vert_{\UU_h^*}$,
    and the above inequality is equivalent to the following inf-sup condition: there exists a positive constant $\beta_{1}$ independent of $h$ such that 
    \begin{equation} \label{eq:Ahinfsup}
        \inf_{\uve_h\in\UU_h \atop \uve_h\neq\Vec{\mathbf{0}}}\sup_{\vv_h\in\UU_h\atop \vve_h\neq\Vec{\mathbf{0}}}\frac{|\langle\mathcal{A}_{h,\mathrm{stab}}(\uve_h),\vve_h\rangle|}{\Vert\uve_h\Vert_{\UU_h}\Vert\vve_h\Vert_{\UU_h}} \geq \beta_{1}.
    \end{equation}
\end{theorem}

\begin{proof}
    Using a decomposition of the form \eqref{eq:oper-decomp}, we express $\mathcal{A}_{h,\mathrm{stab}}$ as follows
    \begin{align*}
		\mathcal{A}_{h,\mathrm{stab}} = 
		\begin{pmatrix}
			\mathcal{A}_{h,\mathrm{sub}} &   \\  & \mathcal{A}_3+\mathcal{S}_h
		\end{pmatrix}
		+
		\begin{pmatrix}
			& \mathcal{C}_{h,1}  \\ \mathcal{C}_{h,2} & 
		\end{pmatrix}
		.\,
	\end{align*}
    Reviewing the proofs of Lemmas~\ref{lemma:bijective-Asub} (Step~2) and~\ref{lemma:InjecA}, one can observe that 
    utilizing the operator $\mathbb{T}$ as part of the first argument of the test function
    is the reason why the coupled system does not possess a symmetric/anti-symmetric structure. Consequently, 
    the sesquilinear forms $b_2(\cdot,\cdot)$ and $c_1(\cdot,\cdot)$ cannot be eliminated through their adjoint counterparts
    and they need to be estimated, compare \eqref{eq:T1_inj_1} and \eqref{eq:lem_A_inj_0}. 
    For $p_h\in Q_h$, it follows from \eqref{prop:Ph1} that $(p_h,\ddiv\vv_h)=(p_h,P_h\ddiv\vv_h)$.
    In order to apply this relation, we 
    introduce the $L^2$ projection $P_h$, which maps to $Q_h$, into the definition of $c_{h1}(\cdot,\cdot)$, 
    so that we can apply the same reasoning as for  \eqref{eq:lem_A_inj_0}.  The same motivation holds for 
    $b_{h2}(\cdot,\cdot)$.  
    The $L^2$ projection $P_h$ is a finite-rank operator, and hence compact. 
    Following Step~2 of the proof of Lemma~\ref{lemma:bijective-Asub} and noting that in finite dimensions 
    injection is equivalent to bijection, shows that $\mathcal{A}_{h,\mathrm{sub}}$ is a bijection. 
    
    Since the $\boldsymbol{BR}$ and $P_1$ elements are $H^1$-conforming, inheriting the compact embedding property from $H^1(\Omega)$ into $L^2(\Omega)$, the operators $\mathcal{C}_{h,1}$ and $\mathcal{C}_{h,2}$ are compact with similar arguments as in continuous case.
    
    Obviously, the combination $\mathcal{A}_3+\mathcal{S}_h$ is bijective.
    Thus, following the proof of Lemma~\ref{lemma:InjecA} gives the statement of the theorem. 
    For the inf-sup condition, we refer to Lemma~35.3 in \cite{ErnFinite} and Theorem~3 in \cite{Gatica2009New}.
\end{proof}

Notice that the coefficient $\beta_1$ in the inf-sup condition \eqref{eq:Ahinfsup} is related to the wave number $\omega$.
That is to say, high frequencies may induce instabilities, leading to a pollution effect. This effect appears in the same way as in the Helmholtz and Maxwell equations \cite{Li2026hybrid}.
In our model, the wave number is theoretically constrained, see \eqref{remark:low frequency}. 
The numerical studies will show that the pollution effect can be reduced by mesh refinement, compare 
Figure~\ref{figErrorAboutWaveNumber} below.

\subsection{Convergence analysis}
We assume some regularity condition of the solution to establish a priori error estimates, concretely:
\begin{equation} \label{regularExactSolution}
    \uu\times\ww\times p\times T \in [H^2(\Omega)]^d\times H^1(\ddiv,\Omega)\times H^1(\Omega)\times H^2(\Omega).
\end{equation}
The dependence of the error bound on the coefficient $\lambda$ will be tracked explicitly, since this coefficient 
is connected to volumetric locking. 
\begin{theorem}[Convergence] \label{Th:errorEstimate}
    Let $\uve$ and $\uve_h$ be the solutions of \eqref{weakF} and \eqref{discF}, respectively. Assume the regularity \eqref{regularExactSolution} holds. Then, there exists a positive constant $C$ independent of $h$, such that
    \begin{equation} \label{eq:ErrorEstimate}
        \Vert \uve-\uve_h \Vert_{\UU_h} \leq Ch\(\vert\uu\vert_2 + \lambda\vert\nabla\cdot\uu\vert_1 + \vert \ww\vert_1 
        + \vert \ddiv\ww\vert_1 + \vert p\vert_1 + \| T\|_2\).
    \end{equation}
\end{theorem}
\begin{proof}
    Let us introduce the notations
    \begin{eqnarray*}
        \vec{\boldsymbol e}_{\uve} &:=& \uve-\uve_h=\(\uu-\uu_h,\ww-\ww_h,p-p_h,T-T_h\), \\
        \vec{\boldsymbol\eta}_{\uve} &:=& \uve-\Ih\uve=\(\uu-\Pi^v_h\uu,\ww-\Pi_h^z\ww,p-P_h p,T-\Pi_h^s T\), \\
        \vec{\boldsymbol\xi}_{\uve} &:=& \Ih\uve-\uve_h=\(\uu_h-\Pi_h^v\uu,\ww_h-\Pi_h^z\ww,p_h-P_h p,T_h-\Pi_h^s T\).
    \end{eqnarray*}
    Thus, the error is decomposed as usual in an interpolation error and a discrete part, 
    $\vec{\boldsymbol e}_{\uve}=\vec{\boldsymbol\eta}_{\uve}-\vec{\boldsymbol\xi}_{\uve}$. 
    
    From the approximation and commutative properties \eqref{prop:Piv2}, \eqref{prop:Piz2}-\eqref{prop:comm1}, and  \eqref{prop:Pis}, one gets
    \begin{eqnarray}
        \Vert\vec{\boldsymbol\eta}_{\uve}\Vert_{\UU_h} &\leq& (2\mu)^{\nicefrac12}\nnormZero{\varepsilon(\uu-\Pi_h^v\uu)} + \lambda^{\nicefrac12}\nnormZero{P_h\ddiv(\uu-\Pi_h^v\uu)} + \Vert\ww-\Pi_h^z\ww\Vert_{\ddiv} \nonumber \\
        && + \nnormZero{p-P_h p} + \nnormOne{T-\Pi_h^s T} +  \(\delta\sum_{K\in\mathcal{T}_h} h_K\Vert\nabla (T-\Pi_h^s T)\Vert^2_{0,K}\)^{\nicefrac12} \label{eq:etaAppro}\\
        &\leq& Ch ( \vert\uu\vert_2 + \vert\ww\vert_1 + \vert\ddiv\ww\vert_1 + \vert p\vert_1 + \vert T\vert_2 ). \nonumber
    \end{eqnarray}
    
    Using the inf-sup condition \eqref{eq:Ahinfsup}, the linearity of $\mathcal{A}_{h,\mathrm{stab}}(\cdot)$, and the triangle inequality, we obtain
    \begin{eqnarray}\label{eq:err1}
        \beta_1\Vert \vec{\boldsymbol\xi}_{\uve} \Vert_{\UU_h} 
        &\leq& \sup_{\vve_h\in\UU_h \atop \vve_h\neq\bzero}\frac{|\langle\mathcal{A}_{h,\mathrm{stab}}(\vec{\boldsymbol\xi}_{\uve}),\vve_h\rangle|}{\Vert\vve_h\Vert_{\UU_h}}\nonumber\\
        &\leq& \sup_{\vve_h\in\UU_h \atop \vve_h\neq\bzero}\frac{|\langle\mathcal{A}_{h,\mathrm{stab}}(\vec{\boldsymbol\eta}_{\uve}),\vve_h\rangle| + |\langle\mathcal{A}_{h,\mathrm{stab}}(\vec{\boldsymbol e}_{\uve}),\vve_h\rangle|}{\Vert\vve_h\Vert_{\UU_h}}.
    \end{eqnarray}
    A straightforward calculation, using \eqref{weakF} and the definition of the sesquilinear forms, gives   
    \begin{eqnarray*}
        \lefteqn{|\langle\mathcal{A}_{h,\mathrm{stab}}(\vec{\boldsymbol e}_{\uve}),\vve_h\rangle|}\\
        &=&|\langle\mathcal{A}_{h,\mathrm{stab}}(\uve),\vve_h\rangle-\langle\mathcal{A}_{h,\mathrm{stab}}(\uve_h),\vve_h\rangle|
         =|\langle\mathcal{A}_{h,\mathrm{stab}}(\uve),\vve_h\rangle-\langle\mathcal{A}(\uve),\vve_h\rangle| \\
        &\leq& |\lambda(\ddiv\uu,\ddiv\vv_h)-\lambda(P_h\ddiv\uu,P_h\ddiv\vv_h)| + |\beta(T,\ddiv\vv_h-P_h\ddiv\vv_h)| \\
        &&+ |i\beta(\ddiv\uu-P_h\ddiv\uu,s_h)| + \delta
        \left| \sum_{K\in\mathcal{T}_h} h_K^2(\nabla T, \nabla s_h)_K\right| =: \sum_{i=1}^4\Psi_i.
    \end{eqnarray*}
    A combination of \eqref{prop:Ph1} and \eqref{prop:Ph2}, in conjunction with the Cauchy--Schwarz inequality, yields 
    \begin{eqnarray*}
        \Psi_1 & = & |\lambda(\ddiv\uu-P_h\ddiv\uu,\ddiv\vv_h) + \lambda(P_h\ddiv\uu,\ddiv\vv_h-P_h\ddiv\vv_h)| \\
        &=& |\lambda(\ddiv\uu-P_h\ddiv\uu,\ddiv\vv_h)| \leq Ch \lambda\vert\ddiv\uu\vert_1 \vert\vv_h\vert_1.
    \end{eqnarray*}
    Due to the fact that $S \subset Q$, $P_hT$ is well defined, so that applying  \eqref{prop:Ph1} gives 
    \[
        \Psi_2 = |\beta(T-P_hT,\ddiv\vv_h-P_h\ddiv\vv_h)|
        = |\beta(T-P_hT,\ddiv\vv_h)| 
        \leq Ch \vert T\vert_1 \vert\vv_h\vert_1.
    \]
    With the Cauchy--Schwarz inequality and \eqref{prop:Ph2}, one obtains 
    \[
     |i\beta(\ddiv\uu-P_h\ddiv\uu,s_h)|  \le \beta \|\ddiv\uu-P_h\ddiv\uu\|_0 \|s_h\|_0 \le C h   \vert\ddiv\uu\vert_1 \|s_h\|_0 .
    \]
    Finally, the Cauchy--Schwarz inequality leads to 
    \begin{eqnarray*}
        \Psi_4 &\leq& 
        \left(\delta\sum_{K\in\mathcal{T}_h} h_K^2 \Vert\nabla T\Vert_{0,K}^2\right)^{\nicefrac{1}{2}}
        \left(\delta\sum_{K\in\mathcal{T}_h} h_K^2 \Vert\nabla s_h\Vert_{0,K}^2\right)^{\nicefrac{1}{2}} \\
        &\leq& Ch \vert T\vert_1 \left(\delta\sum_{K\in\mathcal{T}_h} h_K^2 \Vert\nabla s_h\Vert_{0,K}^2\right)^{\nicefrac{1}{2}}.
    \end{eqnarray*}
From the discrete norm \eqref{norm:disc} and Korn's inequality, we have
    \begin{equation*}
        |\langle\mathcal{A}_{h,\mathrm{stab}}(\vec{\boldsymbol e}_{\uve}),\vve_h\rangle| 
        \leq Ch \(\lambda\vert\ddiv\uu\vert_1 + \vert T\vert_1 \)\Vert\vve_h\Vert_{\UU_h}.
    \end{equation*}
    Inserting this estimate in \eqref{eq:err1}, combining with the interpolation estimate \eqref{eq:etaAppro}, implies \eqref{eq:ErrorEstimate}.
\end{proof}
From the stability of the model problem in the time domain \cite[Theorem 6.1]{JCPBonettiNumerical}, 
we might assume that $\lambda\vert\ddiv\uu\vert_1$ is bounded by the quantities $\tilde{\ff}$, $\tilde{\gf}$, and $\tilde{H}$, thereby obtaining locking-free property with respect to $\lambda$.

	\section{Numerical Studies} \label{sec:numericalExperiments}

This section presents numerical studies to demonstrate the effectiveness and robustness, with respect to large values of $\lambda$,
of the proposed finite element formulation in Section~\ref{sec:discrete-analysis}.  
We divide it into two subsections: the first one verifies the orders of convergence and parameter robustness with an analytic solution, 
while the second one presents results for two benchmark problems.
Note that the degrees of freedom in the frequency-domain problem double, as both the real and imaginary parts of the variables must be solved simultaneously.

\subsection{Convergence and parameter robustness studies} \label{sec:ana_sol}
Let $\Omega=(0,1)^2$. 
A sequence of uniform triangular grids $\{\mathcal{T}_h\}$ is used, where the coarsest grid is obtained by dividing the unit
square with a diagonal from $(0,1)$ to $(1,0)$. 
The spatial mesh size $h$ is defined by dividing the diameter of the mesh cells by $\sqrt{2}$, ranging from $1/8$ to $1/128$. 
The prescribed solution is given by
\begin{flalign}
    \label{Ex:exactSolution}
    \renewcommand{\arraystretch}{1.2}
    \left.
    \begin{array}{l}
        \re u_1=\sin(2\pi y)(-1+\cos(2\pi x))+\frac{1}{\mu+\lambda}\sin(\pi x)\sin(\pi y), \\
        \re u_2=\sin(2\pi x)(1-\cos(2\pi y))+\frac{1}{\mu+\lambda}\sin(\pi x)\sin(\pi y), \\
        \im u_1=\sin(\pi x)\cos(\pi y)+\frac{x^2}{2\lambda}, \\
        \im u_2=-\cos(\pi x)\sin(\pi y)+\frac{y^2}{2\lambda}, \\
        \re w_1=\sin(\pi x)\sin(\pi y), \quad \im w_1=e^y x, \\
        \re w_2=\sin(\pi x)\sin(\pi y), \quad \im w_2=e^x y, \\
        \re p=\sin(\pi x)\sin(\pi y),   \quad \im p=x(1-x)y(1-y), \\
        \re T=x(1-x)y(1-y),         \quad \im T=\sin(\pi x)\sin(\pi y).
    \end{array}\right\}
\end{flalign}
The solid displacement is taken to satisfy $\nabla\cdot\re\uu\rightarrow 0$ and $\nabla\cdot\im\uu\rightarrow 0$ when the Lam\'{e} constant $\lambda\rightarrow\infty$, in order to verify the locking-free property.
The right-hand side terms are determined by inserting the prescribed solution in \eqref{eq:TherPoro-fft}.
We impose pure Dirichlet boundary conditions for all variables directly from \eqref{Ex:exactSolution},  
that is, $\Gamma_d=\Gamma_f=\Gamma_r=\partial\Omega$.

We measure the error of the solid displacement in the $H^1$-seminorm, the error of the filtration displacement in the $H(\ddiv)$-norm, 
the error of the pressure in the  $L^2$-norm, and the error of the temperature in the  $H^1$-seminorm to verify the numerical analysis.
If not mentioned otherwise, 
the coefficients appearing in \eqref{eq:TherPoro-fft}, taken from \cite{JCPBonettiNumerical}, are taken as given in Table~\ref{Tab:modelCoefficients}
and the stabilization parameter appearing in \eqref{eq:stabilization} is set to $\delta=0.1$.

\begin{table}[t!]
\centering
\caption{Problem parameters for the examples from Section~\ref{sec:numericalExperiments}. In the numerical simulations, we used the given values 
of the parameters as their dimensionless counterpart.}
\label{Tab:valueModelCoefficients}
\begin{tabular}{llllll}
\toprule  
Coefficient                &Value     &Coefficient                    &Value      
&Coefficient                     &Value    \\ 
\midrule  
$a_0\ [\unitfrac{Pa}{K^2}]$  &0.2      &$\alpha$\ [-]                                 &1     
&$\boldsymbol{\Theta}\ [\unitfrac{m^2\,Pa}{(K^2\,s)}]$     &$\boldsymbol{I}$              \\
$b_0\ [\unitfrac1{K}]$       &0.1      &$\beta\ [\unitfrac{Pa}{K}]$                   &0.8      
&$\phi$\ [-]                                              &0.5            \\
$c_0\ [\unitfrac1{Pa}]$      &0.2      &$\rho_s\ [\unitfrac{kg}{m^3}]$                &0.03        
&$a$\ [-]                                                 &1              \\
$\lambda\ [\unit{Pa}]$      &1        &$\rho_f\ [\unitfrac{kg}{m^3}]$                &0.03 
&$\tau$\ [\unit{s}]                                       &$1.5\cdot 10^{-2}$ \\
$\mu$\ [\unit{Pa}]           &1        &$\boldsymbol{K}\ [\unitfrac{m^2}{(Pa\,s)}]$   &$\boldsymbol{I}$                      
&$\omega\ [\unitfrac{rad}{s}]$  &1\\
\bottomrule 
\end{tabular}
\end{table}

First, the case of a large Lam\'{e} constant is considered. To this end, $\lambda=10^6$ was chosen. 
As shown in Table~\ref{Tab:lambda-1e6}, our finite element scheme \eqref{discF} is locking-free with respect to 
$\lambda$ while maintaining optimal convergence rates.

\begin{table}[t!]
    \small
    \centering
    \caption{Analytic solution. Numerical results for $\lambda=10^6$ and all other parameter values as listed in Table~\ref{Tab:valueModelCoefficients}.}  \label{Tab:lambda-1e6}.
    \setlength{\tabcolsep}{1mm}\begin{tabular}{@{}ccccccccc@{}}
    \toprule
    $h$ & $\|\nabla (\uu-\uu_h)\|_0$     & Rate  & $\|\ww-\ww_h\|_{\ddiv}$   & Rate    & $\|p-p_h\|_0$ & Rate   & $\|\nabla(T-T_h)\|_0$   & Rate \\ 
    \midrule
        1/8      &1.348e+0      &--     &2.595e-1              &--       &6.536e-2    &--      &4.329e-1              &--  \\
        1/16     &6.629e-1      &1.02   &1.304e-1              &0.99     &3.277e-2    &1.00    &2.181e-1              &0.99  \\
        1/32     &3.300e-1      &1.01   &6.529e-2              &1.00     &1.640e-2    &1.00    &1.092e-1              &1.00  \\
        1/64     &1.648e-1      &1.00   &3.265e-2              &1.00     &8.199e-3    &1.00    &5.465e-2              &1.00  \\
        1/128    &8.238e-2      &1.00   &1.633e-2              &1.00     &4.100e-3    &1.00    &2.733e-2              &1.00 \\
    \bottomrule
    \end{tabular}
\end{table}

We further investigate the sensitivity to parameters related to thermal and storage effects.
Table~\ref{Tab:a0b0c0} presents results for the degenerate coefficients $a_0=b_0=c_0=0$ and one can see that 
the scheme preserves the optimal convergence behavior, demonstrating the robustness with respect to singular parameter values
for $a_0, b_0, c_0$.

\begin{table}[t!]
    \small
    \centering
    \caption{Analytic solution. Numerical results for $a_0=b_0=c_0=0$ and  all other parameter values as listed in Table~\ref{Tab:valueModelCoefficients}.}  \label{Tab:a0b0c0}
    \setlength{\tabcolsep}{1mm}\begin{tabular}{@{}ccccccccc@{}}
    \toprule
    $h$ & $\|\nabla (\uu-\uu_h)\|_0$     & Rate  & $\|\ww-\ww_h\|_{\ddiv}$   & Rate    & $\|p-p_h\|_0$ & Rate   & $\|\nabla(T-T_h)\|_0$   & Rate \\ 
    \midrule
        1/8      &1.339e+0      &--     &2.910e-1              &--       &6.536e-2    &--      &4.329e-1              &--    \\
        1/16     &6.578e-1      &1.03   &1.479e-1              &0.98     &3.277e-2    &1.00    &2.181e-1              &0.99  \\
        1/32     &3.273e-1      &1.01   &7.431e-2              &0.99     &1.640e-2    &1.00    &1.092e-1              &1.00  \\
        1/64     &1.635e-1      &1.00   &3.720e-2              &1.00     &8.199e-3    &1.00    &5.465e-2              &1.00  \\
        1/128    &8.171e-2      &1.00   &1.861e-2              &1.00     &4.100e-3    &1.00    &2.733e-2              &1.00  \\
    \bottomrule
    \end{tabular}
\end{table}

Next, different wave numbers $\omega\in \{ 1, 5, 10, \dots, 50\}$ are studied. Table~\ref{Tab:omega} 
presents exemplarily results for $\omega=25$, which is larger than the critical value \eqref{remark:low frequency}. 
Each variable achieves the optimal convergence order.
Comparing with Tables~\ref{Tab:lambda-1e6} and \ref{Tab:a0b0c0}, it can be seen that the errors for all variables are usually
of the same magnitude. Only the pressure on very coarse grids is more inaccurate in Table~\ref{Tab:omega}.
Figure~\ref{figErrorAboutWaveNumber} illustrates the behavior of the errors with respect to the mesh size $h$ and the wave number $\omega$, to 
provide more details. It can be observed that for $h=1/8$  the errors increase progressively with the wave number.
In contrast, the errors are almost constant for the considered set of wave numbers when $h=1/64$.

\begin{table}[t!]
    \small
    \centering
    \caption{Analytic solution.  Numerical results for $\omega=25$ and all other parameter values as listed in Table~\ref{Tab:valueModelCoefficients}.}  \label{Tab:omega}
    \setlength{\tabcolsep}{1mm}\begin{tabular}{@{}ccccccccc@{}}
    \toprule
    $h$ & $\|\nabla (\uu-\uu_h)\|_0$     & Rate  & $\|\ww-\ww_h\|_{\ddiv}$   & Rate    & $\|p-p_h\|_0$ & Rate   & $\|\nabla(T-T_h)\|_0$   & Rate \\ 
    \midrule
        1/8      &1.351e+0      &--     &3.180e-1              &--        &1.975e-1     &--      &4.733e-1              &--    \\
        1/16     &6.593e-1      &1.03   &1.514e-1              &1.07      &5.680e-2     &1.80    &2.238e-1              &1.08  \\
        1/32     &3.275e-1      &1.01   &7.475e-2              &1.02      &2.008e-2     &1.50    &1.100e-1              &1.02  \\
        1/64     &1.635e-1      &1.00   &3.726e-2              &1.00      &8.696e-3     &1.21    &5.474e-2              &1.01  \\
        1/128    &8.171e-2      &1.00   &1.861e-2              &1.00      &4.163e-3     &1.06    &2.734e-2              &1.00  \\
    \bottomrule
    \end{tabular}
\end{table}

\begin{figure}[t!]
      \centering
      \includegraphics[width=0.8\textwidth]{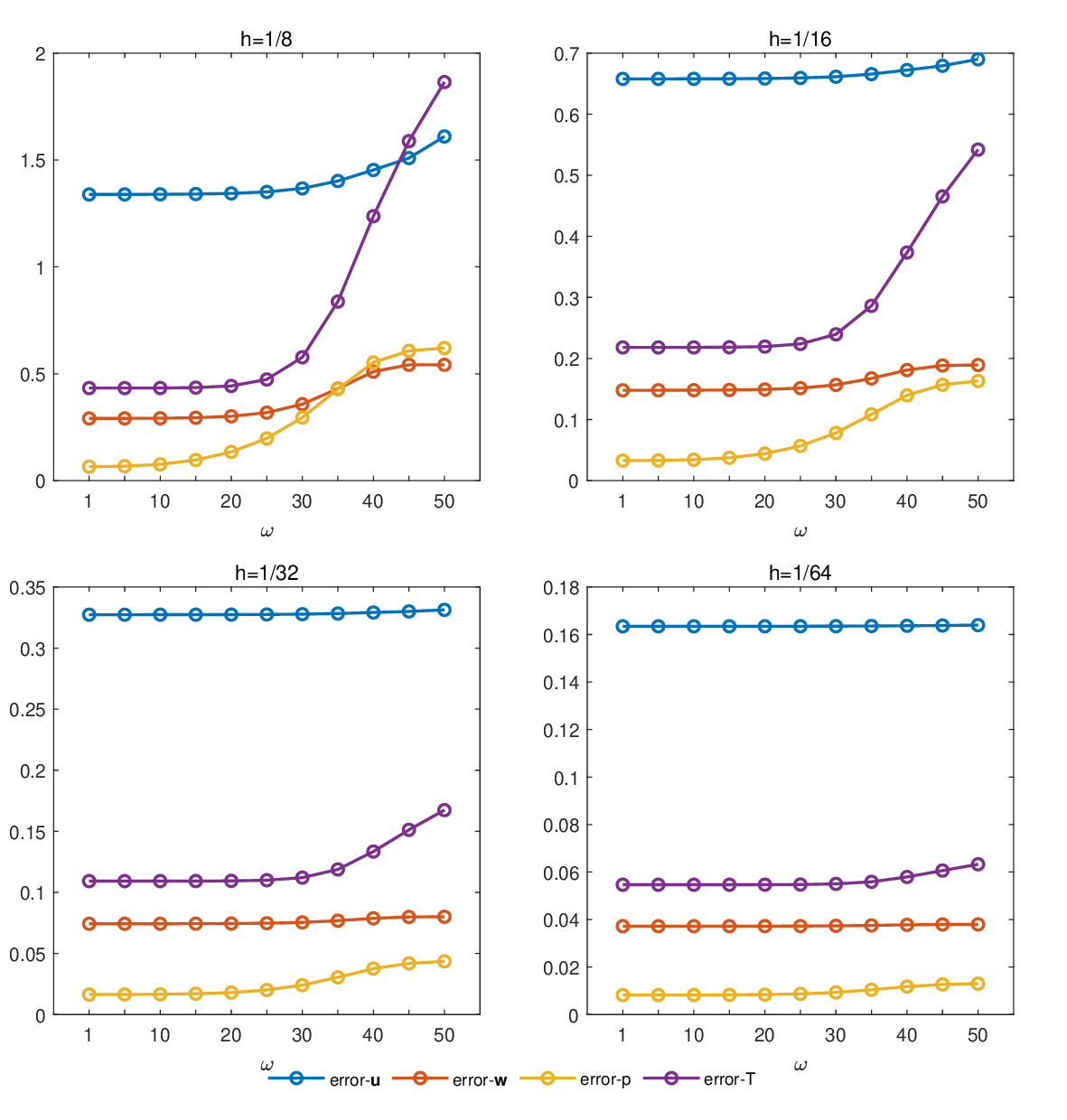}
      \caption{Analytic solution.  Numerical errors for different values of $\omega$ and different mesh sizes $h$.}
      \label{figErrorAboutWaveNumber}
\end{figure}

We finally examine the effect of the temperature stabilization. 
It is shown in Table \ref{Tab:delta0} that, for general values of parameters as in Table \ref{Tab:valueModelCoefficients}, the temperature stabilization is not needed. 
Then, we test the sensitivity of the thermal conductivity to the stabilization, especially when the thermal conductivity is low.
The Lam\'{e} constants are taken as 
\begin{equation}\label{eq:lameMu}
    \lambda=\frac{E\nu}{(1+\nu)(1-2\nu)},\quad \mu=\frac{E}{2(1+\nu)},
\end{equation}
where the elasticity modulus $E=10\ \unit{Pa}$, and Poisson's ratio $\nu=0.499$\ [-], that is, the nearly incompressible case.
Consider $a_0=b_0=c_0=0$, $\boldsymbol{K}=k\boldsymbol{I}$ and $k=10^{-4}$, $\tau=0$. 
We set $\boldsymbol{\Theta}=\theta\boldsymbol{I}$ and $\theta \in \{ 10^{-7}, 10^{-6},\dots,1 \}$, $\delta\in \{ 0, 0.001, 0.01, 0.1, 1 \}$.
As shown in Figure~\ref{figEffectOfDelta}, the error of the temperature decreases gradually with increasing $\theta$.
Setting $\delta=0$ and $\theta\ll 1$ leads to a blow-up of the errors. Furthermore, a small value of $\delta$ is sufficient to address the inaccuracies.

\begin{table}[t!]
    \small
    \centering
    \caption{Analytic solution. Numerical results for $\delta=0$ and  all other parameter values as listed in Table~\ref{Tab:valueModelCoefficients}.}  \label{Tab:delta0}
    \setlength{\tabcolsep}{1mm}\begin{tabular}{@{}ccccccccc@{}}
    \toprule
    $h$ & $\|\nabla (\uu-\uu_h)\|_0$     & Rate  & $\|\ww-\ww_h\|_{\ddiv}$   & Rate    & $\|p-p_h\|_0$ & Rate   & $\|\nabla(T-T_h)\|_0$   & Rate \\ 
    \midrule
        1/8      &1.339e+0      &--     &2.910e-1              &--       &6.536e-2    &--      &4.329e-1              &--    \\
        1/16     &6.578e-1      &1.03   &1.479e-1              &0.98     &3.277e-2    &1.00    &2.181e-1              &0.99  \\
        1/32     &3.273e-1      &1.01   &7.431e-2              &0.99     &1.640e-2    &1.00    &1.092e-1              &1.00  \\
        1/64     &1.635e-1      &1.00   &3.720e-2              &1.00     &8.199e-3    &1.00    &5.465e-2              &1.00  \\
        1/128    &8.171e-2      &1.00   &1.861e-2              &1.00     &4.100e-3    &1.00    &2.733e-2              &1.00  \\
    \bottomrule
    \end{tabular}
\end{table}

\begin{figure}[t!]
      \centerline{
      \includegraphics[width=1.0\textwidth]{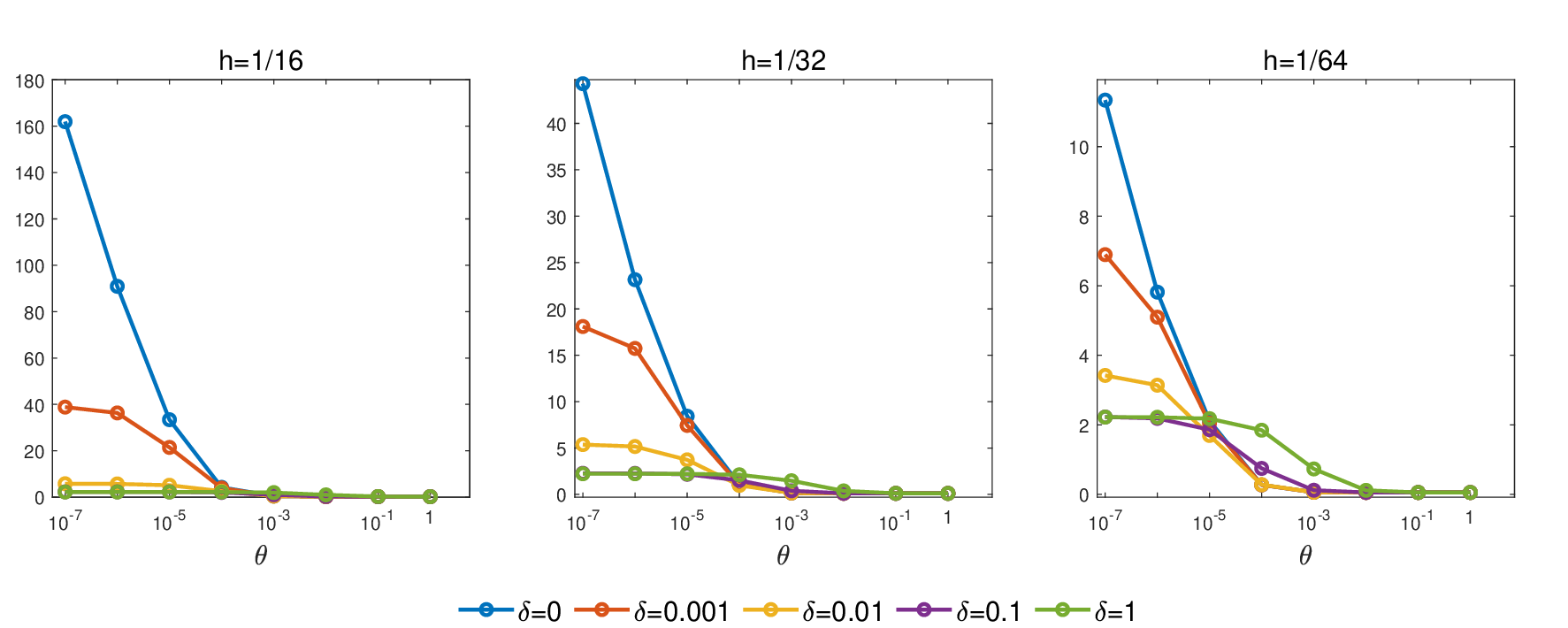}}
      \caption{Analytic solution. Numerical error $\|\nabla(T-T_h)\|_0$ for different values of $\theta$, different mesh sizes $h$, and different values of $\delta$.}
      \label{figEffectOfDelta}
\end{figure}

\subsection{Benchmark problems}

This section considers two benchmark problems, from  \cite{Cristian2024,Yi2017Study}, to demonstrate the elimination of pressure oscillations.

The 2D cantilever bracket problem is shown in Figure~\ref{CBLD} (left). 
The domain is the unit square $\Omega=(0,1)^2$, adiabatic and no-flow boundary conditions are imposed along all sides, that is $\Gamma_p=\Gamma_h=\partial\Omega$.
For the elasticity problem, the left-hand side is clamped, a downward traction $t_N=(0,-1)^T$ 
is applied on the top side, and the right-hand and bottom sides are traction-free.
The material parameters are taken as
given in Table~\ref{Tab:modelCoefficients}, except
$a_0=b_0=c_0=0$, $\delta=0.1$, $\boldsymbol{K}=10^{-7}\boldsymbol{I}$, and $\boldsymbol{\Theta}=10^{-4}\boldsymbol{I}$.
The real and imaginary pressure distributions are depicted in Figure~\ref{CantileverBracket}.
It can be observed that strong tensile and compressive pressures exist at the top and bottom corners of the left boundary.
Our discrete scheme maintains monotonicity without nonphysical pressure oscillations.

\begin{figure}[t!]
   	\centering
   	\includegraphics[width=0.9\textwidth]{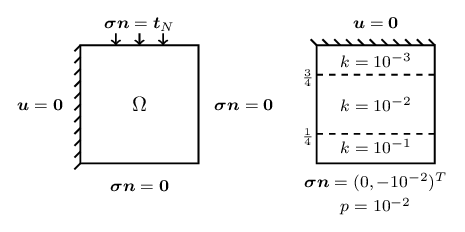}
      \caption{Cantilever bracket problem (left) and layered domain problem (right). }
      \label{CBLD}
\end{figure}
\begin{figure}[t!]
      \centering
      \includegraphics[width=0.9\textwidth]{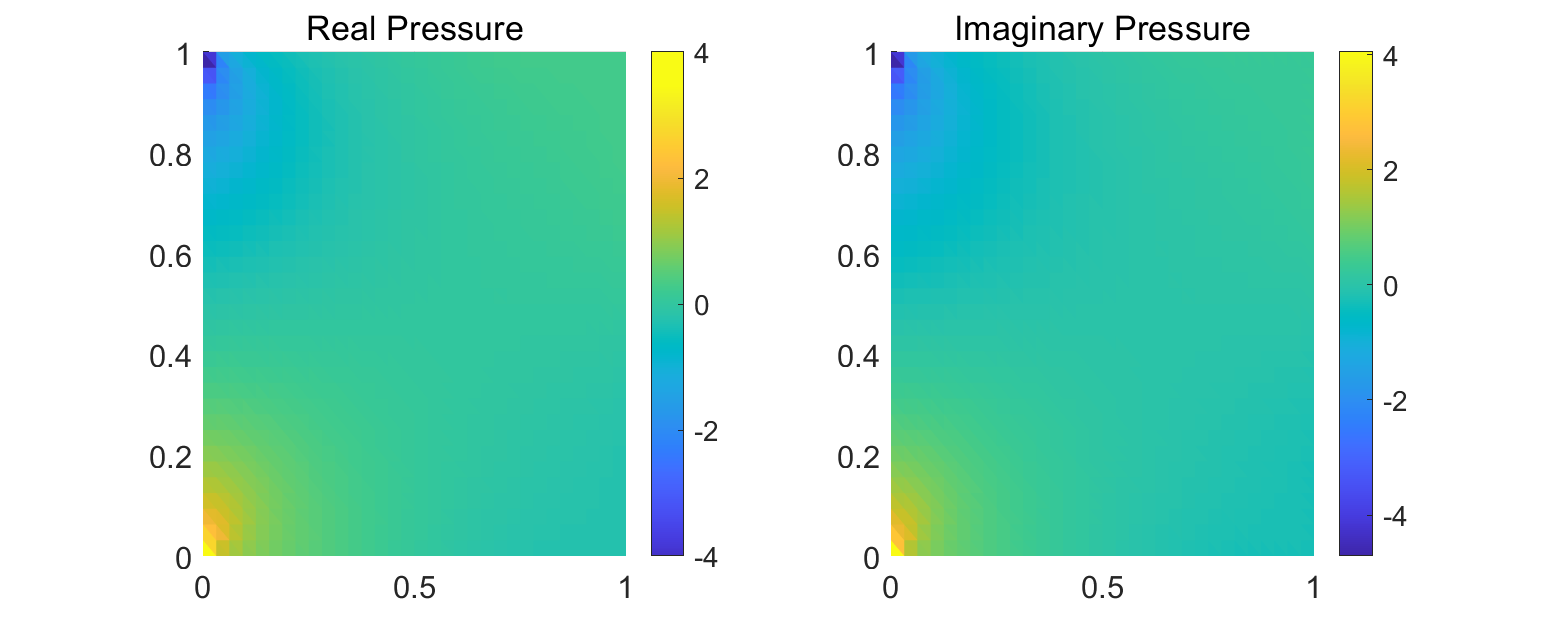}
      \caption{Cantilever bracket problem. Numerical approximations for the real and imaginary part of the pressure. }
      \label{CantileverBracket}
\end{figure}

The second benchmark problem is the layered domain problem as shown in Figure~\ref{CBLD} (right).
We set $\Omega=(0,1)^2$ and decompose it into three subdomains, in which the hydraulic conductivity is set to be 
$\boldsymbol{K}=10^{-1}\boldsymbol{I}$ for $y\in [0,1/4]$, $\boldsymbol{K}=10^{-2}\boldsymbol{I}$ for $y\in [1/4,3/4]$, and $\boldsymbol{K}=10^{-3}\boldsymbol{I}$ for $y\in [0,1/4]$.
Taking $\delta=0.1$, $E=100\ \unit{Pa}$, and $\nu=0.45$, the Lam\'{e} constants $\lambda$ and $\mu$ are obtained by \eqref{eq:lameMu}. 
All other parameters are chosen as in Table~\ref{Tab:valueModelCoefficients}.
We enforce an upward traction on the bottom side of magnitude $10^{-2}\ \unitfrac{dyn}{cm^2}$ and enforce a no-displacement boundary condition on the top side.
A constant pressure of magnitude $10^{-2}\ \unitfrac{dyn}{cm^2}$ is imposed on the bottom side.
The magnitude of pressure for different frequencies is shown in Figure~\ref{LayeredDomainCombine}.
This problem can be simplified vertically to a one-dimensional problem. Thus, we provide the magnitude of pressure at the vertical line $x=0.5$.
It can be seen that the numerical solution is unaffected by small permeability values, and no pressure oscillations are observed at the discontinuity. 

\begin{figure}[t!]
      \centering
      \includegraphics[width=0.9\textwidth]{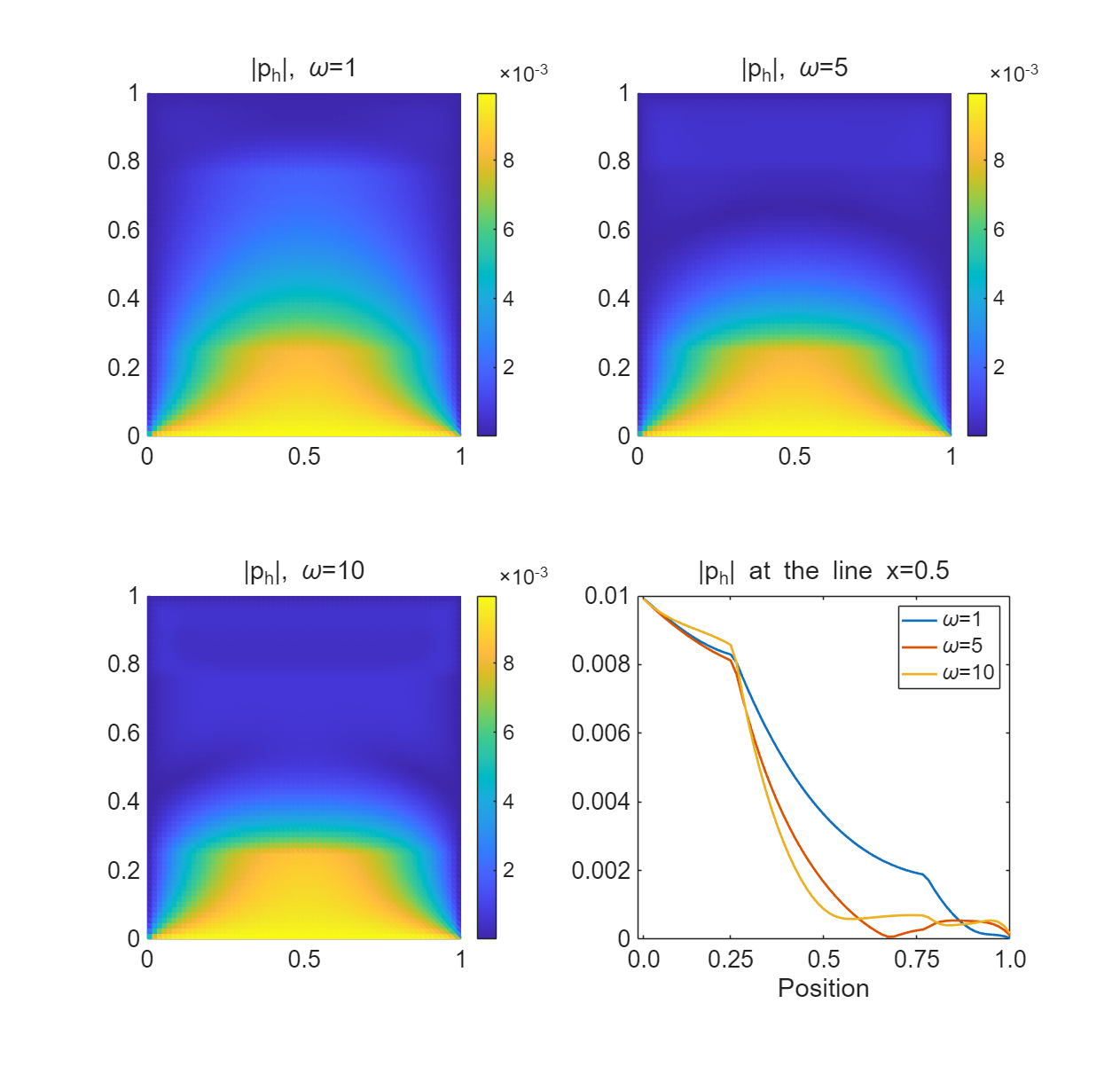}
      \caption{Layered domain problem. Numerical approximations for the magnitude of pressure. }
      \label{LayeredDomainCombine}
\end{figure}

	\section{Summary and Outlook} \label{sec:summary}

In this work, we have proposed and analyzed a finite element approximation to the fully dynamic thermo-poroelasticity problem 
in the frequency domain. Well-posedness for both the continuous and discrete formulations has been established 
by employing the Fredholm alternative and T-coercivity.
The proposed stabilized $\boldsymbol{BR}$-$\boldsymbol{RT}_0$-$P_0$-$P_1$ mixed finite element formulation avoids volumetric locking, pressure and 
temperature oscillations. The optimal order of convergence for all variables was proved. Numerical studies support the analytic results. 

Given that this paper studied a low order finite element method, 
further research will focus on addressing high order cases, their construction and analysis.

	\bibliographystyle{siamplain}
	\bibliography{./references}
	\clearpage
\end{document}